\documentclass[a4paper, abstract=on]{scrartcl}

\usepackage{geometry}
\PassOptionsToPackage{usenames,dvipsnames}{xcolor}
\usepackage[english]{babel}
\usepackage[utf8]{inputenc}
\usepackage{graphicx, latexsym}
\usepackage{amsmath, amssymb, amsthm, amsbsy, mathdots, bm}
\usepackage{framed, multirow, epstopdf, psfrag, hyperref, pdfpages, xspace, microtype, lmodern, paralist, setspace, multicol}
\allowdisplaybreaks{}
\usepackage{algorithm, algorithmicx, algpseudocode}
\algrenewcommand\algorithmicrequire{\textbf{Input:}}
\algrenewcommand\algorithmicensure{\textbf{Output:}}
\algnewcommand\algorithmicassumption{\textbf{Assumption:}}
\algnewcommand\Assumption{\item[\algorithmicassumption]}

\usepackage{blkarray}
\usepackage{float}

\usepackage[noadjust]{cite}
\usepackage{array, csquotes, scrhack, booktabs, breakcites, caption, listings, subcaption}
\usepackage[mathscr]{euscript}
\usepackage{booktabs}

\usepackage{interval}
\intervalconfig{left open fence = (, right open fence =) }

\newcommand{\interor}[2]{\ensuremath{\interval[open right]{#1}{#2}}}

\newcommand{\interc}[2]{\ensuremath{\interval{#1}{#2}}}

\usepackage{cite}
\usepackage{todonotes}

\usepackage[software,hardware]{maxmacros}

\newtheoremstyle{break}     
{0.5em}         
{0.5em}         
{\itshape}      
{}              
{\bfseries}     
{.}             
{\newline}      
{}              
\theoremstyle{break}

\newtheorem{theorem}{Theorem}[section]

\newtheorem{lemma}{Lemma}[section]

\newtheorem{example}{Example}[section]

\newtheorem{remark}{Remark}[section]

\newcommand{\norm}[1]{\left\Vert#1\right\Vert}

\newcommand{\fronorm}[1]{\left\Vert#1\right\Vert_F}
\newcommand{\twonorm}[1]{\left\Vert#1\right\Vert_2}

\newcommand{\frodot}[2]{\langle#1, #2 \rangle_F}

\newcommand{\Rfield}[0]{\mathbb{R}}
\newcommand{\Cfield}[0]{\mathbb{C}}
\newcommand{\Rmat}[2]{\mathbb{R}^{#1 \times{#2}}}

\newcommand{\mbR}[1]{\mathbb{R}^{#1}}

\newcommand{\Nzero}[0]{\mathbb{N}_0}
\newcommand{\Repart}{\operatorname{Re}}

\newcommand{\Cm}{\mathbb{C}^{-}}

\newcommand{\vareps}[0]{\varepsilon}

\newcommand{\spektrum}[1]{\Lambda({#1})}

\DeclareMathOperator{\rank}{rank}

\DeclareMathOperator{\diag}{diag}
\DeclareMathOperator{\range}{range}

\newcommand{\dx}[1]{\mathrm{d}#1}

\newcommand{\rz}[1]{\ensuremath{\mathord{\mathrm{#1}}}}

\newcommand{\mymid}{\;\middle|\;}

\newcommand{\ARE}[0]{\textup{ARE}}
\newcommand{\DRE}[0]{\textup{DRE}}
\newcommand{\NDRE}[0]{\textup{NDRE}}
\newcommand{\ALE}[0]{\textup{ALE}}

\newcommand{\LambdaK}[1]{\ensuremath{\lambda^\downarrow_{{#1}}}}

\newcommand{\Grm}[2]{G^{#1}\left(\Rfield^{#2}\right)}
\newcommand{\GrmG}[2]{G_{0}^{#1}\left(\Rfield^{#2}\right)}

\DeclareMathOperator*{\myargmin}{argmin}

\usepackage{tikz, tikz-3dplot, pgfplotstable, pgfplots}
\pgfplotsset{compat=1.13}
\usetikzlibrary{mindmap,shadows}
\usetikzlibrary{decorations}
\usetikzlibrary{decorations.markings}
\usetikzlibrary{matrix,decorations.pathreplacing,calc}
\usetikzlibrary{patterns}
\pgfplotsset{filter discard warning=true}

\pgfplotsset{
    discard if not/.style 2 args={
        filter discard warning=false,
        x filter/.append code={
            \edef\tempa{\thisrow{#1}}
            \edef\tempb{#2}
            \ifx\tempa\tempb{}
            \else
                
            \fi
        }
    },
    discard if not symbolic/.style 2 args={
        filter discard warning=false,
        x filter/.append code={
            \edef\tempa{\thisrow{#1}}
            \edef\tempb{#2}
            \ifx\tempa\tempb{}
            \else
                
            \fi
        },
    }
}

\newcommand{\LIE}{\texttt{LIE}}
\newcommand{\STRANG}{\texttt{STRANG}}
\newcommand{\SYMMETRICtwo}{\texttt{SYMMETRIC2}}
\newcommand{\SYMMETRICfour}{\texttt{SYMMETRIC4}}
\newcommand{\SYMMETRICsix}{\texttt{SYMMETRIC6}}
\newcommand{\SYMMETRICeight}{\texttt{SYMMETRIC8}}
\newcommand{\RAIL}{\texttt{RAIL}}
\newcommand{\CONVDIFF}{\texttt{CONV\_DIFF}}

\newcommand{\AREGALERKIN}{\texttt{ARE-Galerkin}}

\newcommand{\vspacealgorithm}{\vspace{.1in}}

\author{Maximilian Behr \and Peter Benner \and Jan Heiland}
\title{Invariant Galerkin Ansatz Spaces and Davison-Maki Methods for the Numerical Solution of Differential Riccati Equations}
\date{\today}
\begin{document}
\maketitle
\begin{abstract}
    \noindent

The differential Riccati equation appears in different fields of applied mathematics like control and system theory.
Recently Galerkin methods based on Krylov subspaces were developed for the autonomous differential Riccati equation.
These methods overcome the prohibitively large storage requirements and computational costs of the numerical solution.
In view of memory efficient approximation, we review and extend known solution formulas and identify invariant subspaces for a possibly low-dimensional solution representation.
Based on these theoretical findings, we propose a Galerkin projection onto a space related to a low-rank approximation of the algebraic Riccati equation.
For the numerical implementation, we provide an alternative interpretation of the modified \emph{Davison-Maki method} via the transformed flow of the differential Riccati equation,
which enables us to rule out known stability issues of the method in combination with the proposed projection scheme.
We present numerical experiments for large-scale autonomous differential Riccati equations and compare our approach with high-order splitting schemes.

\end{abstract}

\tableofcontents

\section{Introduction}\label{sec:introduction}

In this paper we consider the autonomous differential Riccati equation
\begin{align*}
    \dot{X}(t) &= A^T X(t) + X(t) A  - X(t) BB^T X(t) +C^T C, \\
    X(0) &= X_0.
\end{align*}
The equation plays an important role in model order reduction, optimal control, differential games and stability analysis~\cite{AboFIJ03,Loc01,KnoK85,Bro70,McC17,BecM15,BenM18}.
We focus in this work on the large-scale case.
In this setting, the numerical approximation of $X$ comes with high memory requirements and high computational costs. Just the storage of the solution at the relevant time instances would scale with ${N_t}n^2$, where $n$ is the dimension of the problem and $N_t$ is the number of time steps.
The approach of first discretizing in time and then focusing on efficient approximation of the resulting algebraic equations has been the main course of research on this problem setup, see, e.g.,~\cite{BenM04,BenM13,Hei14,KoeLS16,LanMS15,Men07,Sti15,Lan17,Sti18,MenPS19,BenM18}.
In all these approaches, the approximation of at least one large-scale algebraic equation has to be solved and stored for every time step so that the memory demands still scale with ${N_t}n$.
Conceptually, it seems more beneficial for the autonomous differential Riccati equation to first reduce the problem dimensions to, say, $k\ll n$ and then approach the reduced equation as this leads to storage requirements in the order of ${N_t}k$. In this respect, Krylov subspace methods have been proposed~\cite{GulHJetal18,HacJ18a,HacJ18b,HacJ18c,KosM17,KiS19,AngHJ18} that generate a trial space for the numerical solution using an Arnoldi method.
The resulting Galerkin projected system is of lower order and can be solved with low memory demand and with various methods that exist for differential Riccati equations of small or moderate size.

In this work, we develop a Galerkin approach, where the trial space is based on the numerical solution of the algebraic Riccati equation. This extends the concepts of our previous work on a numerical scheme for differential Lyapunov equations~\cite{BehBH18}.

The paper is organized as follows.
In Section~\ref{sec:algebraic_and_differential_riccati_equations} we introduce the algebraic and differential Riccati equations and review the relevant fundamental properties about their solutions.
In Section~\ref{sec:randons_lemma} we review \emph{Radon's Lemma} and work out its implication that the differential Riccati equation is connected to a flow on a Grassmanian manifold.
Moreover, in Section~\ref{sec:radons_lemma:solution_formulas}, we apply \emph{Radon's Lemma} to obtain solution formulas for the differential Riccati equation based on the solution of the algebraic Riccati equation that we will use to explain and illustrate the major source of numerical instabilities of the \emph{Davison-Maki method} for the numerical solution of the differential Riccati equation; see Section~\ref{sec:radons_lemma:davison_maki_methods}  Then we will use the connection to the Grassmanian manifold to derive the \emph{modified Davison-Maki method} in a way that overcomes these instabilities.
In Section~\ref{sec:galerkin_approach_for_large_scale_differential_riccati_equations}, we develop a Galerkin approach for the solution of the differential Riccati equation in the matrix exponential representation that results from \emph{Radon's Lemma}.
We combine the monotonicity of the solution of the differential and relevant properties of the solution of the algebraic Riccati equation to define a suitable and numerically computable trial space for the approximation of the solution of the differential Riccati equation.
We propose to solve the resulting Galerkin system with the \emph{modified Davison-Maki method}.
Numerical results are presented in Section~\ref{sec:numerical_experiments} and Appendices~\ref{appendix:numerical_results_for_galerkin_approach} and~\ref{appendix:numerical_results_for_splitting_schemes}.

\section{Preliminaries}\label{sec:preliminiaries}

In this section we set the notation and review some basic results from linear algebra.
The identity matrix and zero matrix of size $n$ are written by $I_{n}$ and $0$.
The image or column space of a matrix $A\in \Rmat{n}{m}$ is denoted by $\range{(A)}$, and its kernel or null space by $\ker(A)$.
The 1--norm, 2--norm, Frobenius norm and Frobenius inner product are denoted by $\norm{\cdot}_1, \norm{\cdot}_2, \norm{\cdot}_F$ and $\frodot{\cdot}{\cdot}$, respectively.
The spectrum of a quadratic matrix $A$ is denoted by  $\spektrum{A}$. Generally, the spectrum is a subset of $\Cfield$.
A matrix is called stable, if its spectrum is contained in the left open complex half plane $\Cm$, i.e. $\spektrum{A}\subseteq \Cm$.
If $A$ is real and symmetric, all eigenvalues are real and $\LambdaK{k}(A)$ represents the $k$--largest eigenvalue.
Therefore, $\LambdaK{1}(A)\geq \LambdaK{2}(A) \geq \cdots \geq \LambdaK{n}(A)$ are the eigenvalues of $A$ ordered in a non-decreasing fashion.
The Loewner partial ordering on the set of real symmetric matrices is defined by $A \preccurlyeq B$, which means $B-A$ is positive semidefinite,~\cite[Ch. 7.7]{HorJ85}.
The orthogonal complement of a linear subspace $U\subseteq \Rfield^{n}$ is denoted by $U^{\perp}\subseteq \Rfield^{n}$.
For $A\in \Rmat{n}{n}$ and $B\in \Rmat{n}{b}$, the image of the Krylov matrix generated by $A$ and $B$ is denoted by
\begin{align*}
    \mathcal{K}(A,B):=\range\left(\begin{bmatrix} B, AB, \ldots, A^{n-1}B\end{bmatrix} \right)\subseteq \Rfield^{n}.
\end{align*}
The linear space $\mathcal{K}(A,B)$ is $A$--invariant.

\section{Algebraic and Differential Riccati Equations}\label{sec:algebraic_and_differential_riccati_equations}

In this section we introduce the algebraic and differential Riccati equation (\ARE/\DRE) and the
algebraic Lyapunov equation (\ALE).

Consider $A, X_0 \in \Rmat{n}{n}$ and $C\in \Rmat{c}{n}$ and $B\in \Rmat{n}{b}$.
Throughout this paper, we assume that $X_0$ is a symmetric positive semidefinite matrix and consider the \DRE{}
\begin{subequations}\label{eqn:ch_adre:auto_dre}
    \begin{align}
        \dot{X}(t) & = \mathcal R(X(t)) :=  A^{T} X(t) + X(t) A - X(t) BB^{T} X(t) + C^{T} C,           \label{eqn:ch_adre:auto_dre_eqn}              \\
        X(0)       & = X_0.                                                                             \label{eqn:ch_adre:auto_dre_initial}
    \end{align}
\end{subequations}

Stationary points of~\eqref{eqn:ch_adre:auto_dre_eqn} are solutions of the corresponding \ARE{}
\begin{align}
    0 & = \mathcal R(X) = A^{T} X + X A - X B B^{T} X + C^{T} C.    \label{eqn:ch_adre:are}
\end{align}
The linear version ($B=0$) of the \ARE{} is the \ALE{}
\begin{align}
    0 & = A^{T} X + X A + C^{T} C.  \label{eqn:ch_adre:ale}
\end{align}

We review some fundamental results about existence, uniqueness and properties of the solution of the \DRE~\eqref{eqn:ch_adre:auto_dre}, \ARE~\eqref{eqn:ch_adre:are} and the \ALE~\eqref{eqn:ch_adre:ale}.

\begin{theorem}[Existence and Uniqueness of Solutions to the \ALE{}~\eqref{eqn:ch_adre:ale},~{\cite[Thm. 1.1.3, 1.1.7]{AboFIJ03}}]\label{thm:ch_adre:ale_existence}
    If $\spektrum{A}\cap \spektrum{-A}= \emptyset$, then the \ALE{}~\eqref{eqn:ch_adre:ale} admits a unique solution $X_L\in \Rmat{n}{n}$.
    The solution $X_L$ is symmetric.
    If $A$ is stable, then $X_L$ is symmetric positive semidefinite and given by
    \begin{align}
        X_L = \int\limits_{0}^{\infty} e^{tA^T}{C}^T Ce^{tA} \dx{t}. \label{eqn:ch_adre:ale_sol_formula}
    \end{align}
\end{theorem}

\begin{theorem}[Existence and Uniqueness of Solutions for the \ARE{}~\eqref{eqn:ch_adre:are},~{\cite[Lem. 2.4.1, Cor. 2.4.3]{AboFIJ03},~\cite[Ch. 10]{KnoK85}}]\label{thm:ch_adre:are_existence}
    Let $(A,B)$ be stabilizable and $(A,C)$ be detectable, then the \ARE{}~\eqref{eqn:ch_adre:are} has a unique stabilizing solution $X_{\infty}\in \Rmat{n}{n}$.
    This means $\mathcal{R}(X_{\infty}) = 0$ and $\Lambda(A-BB^T X_{\infty})\subseteq \Cm$.
    Moreover $X_{\infty}$ is symmetric positive semidefinite and there is no other symmetric positive semidefinite solution of the \ARE{}~\eqref{eqn:ch_adre:are}.
\end{theorem}

\begin{theorem}[Range of the Solution of the \ARE{}~\eqref{eqn:ch_adre:are},~{\cite[Thm. 3.2]{BehBH18b}}]\label{thm:ch_adre:are_range}
    Let $(A,B)$ be stabilizable, $(A,C)$ be detectable and $X_{\infty} \in \Rmat{n}{n}$ be the unique stabilizing solution of the \ARE{}~\eqref{eqn:ch_adre:are}.
    Then the following relation holds:
    \begin{align*}
        \range\left(X_{\infty}\right) = \mathcal{K}\left(A^T, C^T\right).
    \end{align*}
\end{theorem}
The inclusion $\mathcal{K}\left(A^T, C^T\right) \subseteq \range{ \left(X_{\infty}\right)}$ in Theorem~\ref{thm:ch_adre:are_range} is actually true for each symmetric solution of the ARE~\eqref{eqn:ch_adre:are}, cf.~\cite[Lemma 2.4.9]{AboFIJ03}.
In~\cite[Ch. 3.3]{AmoB10} a Kalman decomposition is used to show that $\rank\left( X_{\infty} \right) = \dim\left( \mathcal{K}\left(A^T, C^T\right) \right)$.
A connection between the space $\mathcal{K}\left(A^T,C^T\right)$ and a certain Krylov subspace generated by the associated Hamiltonian matrix, which can be used for numerical approximation of the solution of the ARE~\eqref{eqn:ch_adre:are},
was presented in~\cite[Thm. 10]{BenB16}.

Typically, solutions of quadratic differential equations like the \DRE~\eqref{eqn:ch_adre:auto_dre} exhibit a finite-time escape phenomena.
By means of comparison arguments and the fact that $-BB^T$ is negative semidefinite one can show that the solution exists for all $t\geq0$.
With additional assumptions, the solution converges monotonically to the unique solution of \ARE~\eqref{eqn:ch_adre:are} and is, thus, bounded.


\begin{theorem}[Existence and Uniqueness of Solutions of the \DRE{}~\eqref{eqn:ch_adre:auto_dre},~{\cite[Thm. 4.1.6, 4.1.8]{AboFIJ03},~\cite[Ch. 10]{KnoK85}}]\label{thm:ch_adre:dre_existence}
    The \DRE~\eqref{eqn:ch_adre:auto_dre} has a unique solution $X\colon \interor{0}{\infty}\to \Rmat{n}{n}$. The solution $X$ has the following properties:
    \begin{itemize}
        \item $X(t)$ is symmetric positive semidefinite for all $t \geq 0$.
        \item If $\dot{X}(0)=\mathcal R(X_0)\succcurlyeq 0$, then $t\mapsto X(t)$ is monotonically non-decreasing on $\interor{0}{\infty}$,
            i.e. \mbox{$X(t_1)\preccurlyeq X(t_2)$} for all $t_1,t_2$ such that $0\leq t_1\leq t_2$.
    \end{itemize}
\end{theorem}

\begin{theorem}[Invariant Subspace of the Solution of the \DRE{}~\eqref{eqn:ch_adre:auto_dre}, cp.~{\cite[Thm. 3.1]{BehBH18b}}]\label{thm:ch_adre:dre_invariant_subspace}
    Let the columns of $Q\in \Rmat{n}{p}$ span an orthonormal basis of $\mathcal{K}\left(A^T, C^T\right)$ and define
    the linear space $\mathcal{Q}:= \left\{ Q Y Q^T \mid Y\in \Rmat{p}{p}\right\}\subseteq \Rmat{n}{n}$ or $\mathcal{Q}:=\{0\}\subseteq \Rmat{n}{n}$, if $C=0$. Then the following holds:
    \begin{align*}
        X(t) \in \mathcal{Q}\ \text{for all}\ t\geq 0,
    \end{align*}
    where $X$ is the unique solution of the~\DRE{}~\eqref{eqn:ch_adre:auto_dre} with $X_0 = 0$.
\end{theorem}

With this relation, one can readily confirm that the solution of the \DRE{}~\eqref{eqn:ch_adre:auto_dre} evolves in an invariant subspace of $\Rmat{n}{n}$.

For numerical approximations of the solutions of large-scale \ALE{s}, \ARE{s} and \DRE{s}, one typically seeks for low-rank approximations, i.e.\ a $q\ll n$  so that the relation in Theorem~\ref{thm:ch_adre:dre_invariant_subspace} is still valid up to a given tolerance, to avoid overly demanding memory requirements.
Therefore, the relevant literature features numerous contributions which study the decay rate of $\LambdaK{k}(X)$ or $\frac{\LambdaK{k}(X)}{\LambdaK{1}(X)}$ for increasing $k$; see, e.g.,~\cite{BecT17,BakES15,AntSZ02,GruK14,Pen00,Opm15,Gra04,SorZ02} on the eigenvalue decay of the solution of the \ALE{} and~\cite{Opm15,BenB16,Sti18} for results on the \ARE{} and \DRE{}.

For the autonomous \DRE{}~\eqref{eqn:ch_adre:auto_dre}, one can derive estimates based on the monotonicity. Assume that $\mathcal{R}(X_0)\succcurlyeq 0$, then by Theorem~\ref{thm:ch_adre:dre_existence} the function $t\mapsto X(t)$ is monotonically non-decreasing on $\interor{0}{\infty}$, where $X$ is the unique
solution of the \DRE{}~\eqref{eqn:ch_adre:auto_dre}.
A direct consequence of the Courant-Fischer-Weyl min-max principle~\cite[Cor. 7.7.4]{HorJ85} implies that $t\mapsto \LambdaK{k}(X(t))$ is also monotonically non-decreasing on $\interor{0}{\infty}$.
Therefore the number of eigenvalues of $X(t)$ greater than or equal to a given threshold $\vareps >0$ is non-decreasing over time.

\begin{example}[Eigenvalue Decay]\label{ex:eigenvalue_decay}
    We illustrate this by an example in \textup{Figure~\ref{fig:ch_adre:eig_decay_dre_time}}.
We have chosen $X_0 = 0$, $C=\begin{bmatrix} 1,\ldots, 1 \end{bmatrix}=B^T$ and $A$ to be tridiagonal with entries $5, -1, -5$ on the subdiagonal, diagonal and superdiagonal, respectively.
The matrices are of size $n=100$ and the \DRE{} was solved numerically to a high precision on the time interval $\interc{0}{15}$.
For this we have used the variable-precision arithmetic \textup{\texttt{vpa}} of \textup{\matlab{} 2018a} with $512$ significant digits and \textup{Algorithm~\ref{alg:nonsymdre:mod_davison_maki}} with step size $h=2^{-5}$.
The eigenvalues of $X(t)$ are arranged in a non-increasing order and plotted for $t\in \{0.5, 1, \ldots, 15\}$. The functions $t\mapsto \LambdaK{k}(X(t))$ are highlighted in red for $k\in \{10,20,30,40,50\}$.
All eigenvalues below $10^{-60}$ were truncated from \textup{Figure~\ref{fig:ch_adre:eig_decay_dre_time}}.
The shadowed red plane is drawn at the level $2\cdot10^{-16}$, which is approximately machine precision in double arithmetic.

\begin{figure}[H]
    \centering
    \includegraphics{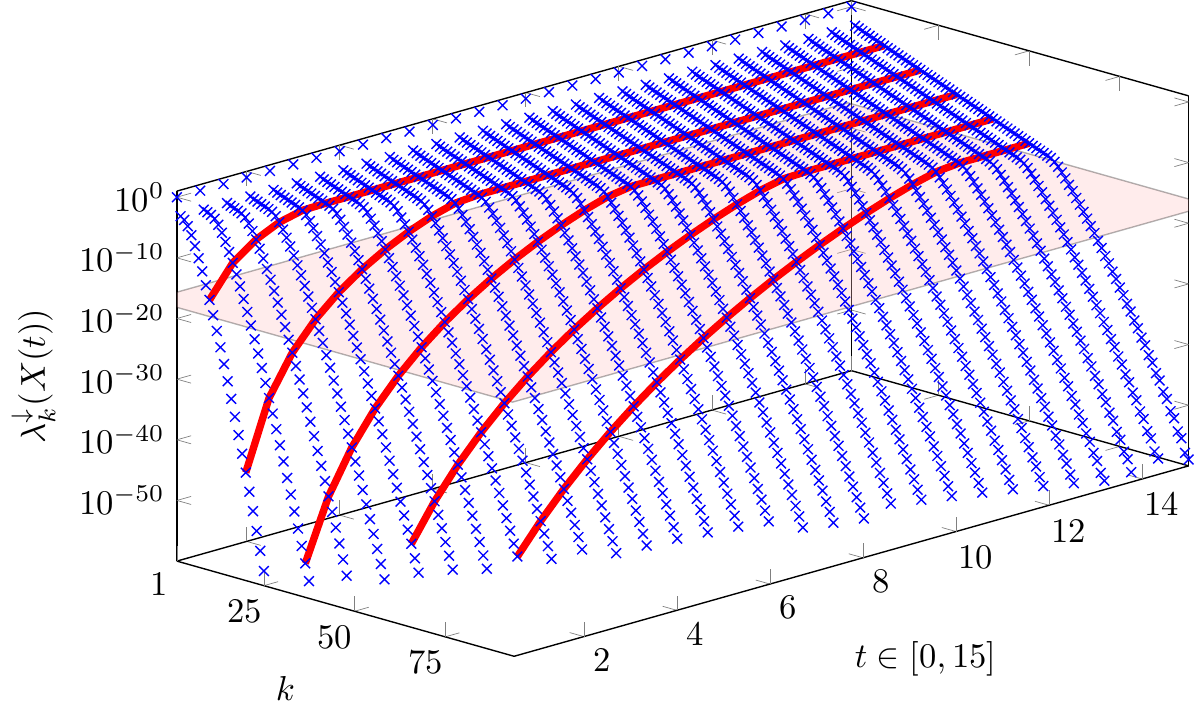}
    \caption{Eigenvalues $\LambdaK{k}(X(t))$ of the numerical solution of \DRE{}~\eqref{eqn:ch_adre:auto_dre}.}\label{fig:ch_adre:eig_decay_dre_time}
\end{figure}

\end{example}

\section{Radon's Lemma}\label{sec:randons_lemma}
In this section, we consider the non-symmetric differential Riccati equation abbreviated by \NDRE{} as a generalization of the \DRE{}.
We will make heavy use of \emph{Radon's Lemma} that shows that the \NDRE{} is locally equivalent to a linear differential equation of twice the size.
Vice versa, the solution of the \NDRE{} defines the solution of an associated linear system.

\emph{Radon's Lemma}~(Thm.~\ref{thm:radon}) has several consequences.
In Section~\ref{sec:radons_lemma:flow_on_the_grassmanian_manifolds}, we review the fact that the solution of the \NDRE{} induces a flow on the Grassmanian manifold.
This flow has a simpler structure as it is based on a matrix exponential.
In Section~\ref{sec:radons_lemma:solution_formulas} we show how solution formulas can be obtained by applying suitable linear transformations, which decouple the
linear differential equation.
Then, in view of numerical approximation, we review the \emph{Davison-Maki method} and the \emph{modified Davison-Maki method} in Section~\ref{sec:radons_lemma:davison_maki_methods}.
We use the solution formula from Section~\ref{sec:radons_lemma:solution_formulas} to explain, why the \emph{Davison-Maki method} applied to the \DRE{} usually suffers from numerical instabilities and show that an exploitation of the structure of the transformed flow on the Grassmanian manifold leads to a suitable modification of the \emph{Davison-Maki method}.  

\begin{theorem}[Radon's Lemma,~{\cite[Thm. 3.1.1]{AboFIJ03}}]\label{thm:radon}
    Let $M_{11}\in \Rmat{n}{n},\ M_{12}\in \Rmat{n}{m},\ M_{21},M_0\in \Rmat{m}{n},\ M_{22}\in \Rmat{m}{m}$ and $\mathbb{I}\subseteq \Rfield$ be an open interval such that $0\in \mathbb{I}$. We consider the \NDRE{}
    \begin{subequations}\label{eqn:nonsymdre}
        \begin{align}
            \dot{W}(t) & = M_{22} W(t) - W(t) M_{11} - W(t) M_{12}(t) W(t) + M_{21},     \label{eqn:nonsymdre_eqn}     \\
            W(0)       & = M_{0}.                                                        \label{eqn:nonsymdre_initial}
        \end{align}
    \end{subequations}
    The following holds:
    \begin{enumerate}
        \item Let $W\colon \mathbb{I} \to \Rmat{m}{n}$ be the solution of~\eqref{eqn:nonsymdre} and $U\colon \mathbb{I} \to \Rmat{n}{n}$ be the solution of the linear initial value problem
            \begin{align}
                \dot{U}(t) = \left( M_{11} + M_{12} W(t)\right) U(t),\ U(0)  = I_n. \label{eqn:nonsym_dreU}
            \end{align}
            Moreover let $V(t):=W(t)U(t)$. Then $U\colon \mathbb{I} \to \Rmat{n}{n}$ and $ V \colon \mathbb{I} \to \Rmat{m}{n}$ define the solution of
            \begin{align}
                \begin{bmatrix}\dot{U}(t) \\ \dot{V}(t) \end{bmatrix} =
                M \begin{bmatrix} U(t) \\ V(t) \end{bmatrix} :=
                \begin{bmatrix} M_{11} & M_{12} \\ M_{21} & M_{22} \end{bmatrix} \begin{bmatrix} U(t) \\ V(t) \end{bmatrix},\ \begin{bmatrix}U(0) \\ V(0) \end{bmatrix} =
                \begin{bmatrix} I_{n} \\ M_{0} \end{bmatrix}.                                                                \label{eqn:nonsymhami}
            \end{align}
        \item If $\begin{bmatrix} U \\ V \end{bmatrix} \colon \mathbb{I} \to \Rmat{(n+m)}{n}$ is a solution of~\eqref{eqn:nonsymhami} and the matrix $U(t)$ is nonsingular for all $t\in \mathbb{I}$, then
            $W\colon \mathbb{I} \to \mbR{m\times n},\ W(t) = V(t) {U(t)}^{-1}$ is a solution of~\eqref{eqn:nonsymdre}.
    \end{enumerate}
\end{theorem}
Radon's Lemma~(Thm.~\ref{thm:radon}) also holds for time-dependent continuous matrix valued functions as coefficients.
Note that, usually, the solution of the \NDRE{}~\eqref{eqn:nonsymdre} has finite time escape, while the solution of system~\eqref{eqn:nonsymhami} exists for all $t\in \Rfield$.
However, one can consider the solution $W$ of the \NDRE{}~\eqref{eqn:nonsymdre} on the interval of existence.
As the function $U$ is a solution of the linear initial value problem~\eqref{eqn:nonsym_dreU} and $U(0)=I_n$ is nonsingular, the determinant of $U(t)$ can not vanish on the interval $\mathbb{I}$.
It follows that the matrix $U(t)$ is nonsingular for all $t \in \mathbb{I}$, c.f.~\cite[\S 15]{Wal98}.
Therefore as long as the solution of the \NDRE{}~\eqref{eqn:nonsymdre} exists, it can be recovered from the solution of system~\eqref{eqn:nonsymhami}.

    \subsection{Flow on the Grassmanian Manifold}\label{sec:radons_lemma:flow_on_the_grassmanian_manifolds}
In this section we review the fact that the solution of the \NDRE{}~\eqref{eqn:nonsymdre} is locally equivalent to a flow on the Grassmanian manifold.
This connection was first observed in~\cite{Sch73} and the corresponding flow was further studied in~\cite{Sha86,martin1980grassmannian}.
The content of this subsection is a summary of~\cite[\S 2]{Sha86}.
One main observation from Radon's Lemma~(Thm.~\ref{thm:radon}) is that the solution $W$ of the \NDRE{}~\eqref{eqn:nonsymdre} depends only on the
linear space spanned by $U(t)$ and $V(t)$.
This can be seen by the following arguments.
Let $\begin{bmatrix} U \\ V \end{bmatrix} \colon \mathbb{I} \to \Rmat{(n+m)}{n}$ be a solution of~\eqref{eqn:nonsymhami} and $t_0 \in \mathbb{I}$.
Moreover assume that $\tilde{U}\in \Rmat{n}{n}, \tilde{V}\in \Rmat{m}{n}$ are such that
\begin{align*}
    \range\left( \begin{bmatrix} \tilde{U} \\ \tilde{V} \end{bmatrix} \right) = \range\left( \begin{bmatrix} U(t_0) \\ V(t_0) \end{bmatrix}\right).
\end{align*}
The linear spaces are equal, if and only if there is a nonsingular matrix $T\in \Rmat{n}{n}$ such that
\begin{align*}
    \begin{bmatrix} \tilde{U} \\ \tilde{V} \end{bmatrix} = \begin{bmatrix} U(t_0) \\ V(t_0) \end{bmatrix} T.
\end{align*}
Since $U(t_0)$ is nonsingular, we have
\begin{align*}
    \tilde{V}\tilde{U}^{-1} = V(t_0) T T^{-1} {U(t_0)}^{-1} = V(t_0) {U(t_0)}^{-1} = W(t_0).
\end{align*}
Consequently, it is the linear subspace $\range\left(\begin{bmatrix} U(t) \\ V(t) \end{bmatrix}\right)\subseteq \Rfield^{n+m}$ that defines the solution $W(t)$, rather than the chosen basis $\begin{bmatrix} U(t) \\ V(t) \end{bmatrix}$ to represent the space.
Since
\begin{align*}
    \range\left(\begin{bmatrix} U(t) \\ V(t) \end{bmatrix}\right) = \range\left( e^{tM} \begin{bmatrix} I_n \\ M_0 \end{bmatrix} \right) = e^{tM}\range\left( \begin{bmatrix} I_n \\ M_0 \end{bmatrix} \right),
\end{align*}
and the (nonsingular) matrix exponential is applied to an $n$-dimensional subspace $\range\left( \begin{bmatrix} I_n \\ M_0 \end{bmatrix} \right)$, we obtain a time-dependent family of $n$-dimensional subspaces of $\Rfield^{n+m}$.
The Grassmanian manifold $\Grm{n}{n+m}$ consists of  all $n$-dimensional subspaces of $\Rfield^{n+m}$.
Therefore the flow associated to the \NDRE{}~\eqref{eqn:nonsymdre} on $\Grm{n}{n+m}$ is given by
\begin{align*}
    \varphi&: \Rfield \times \Grm{n}{n+m} \to \Grm{n}{n+m},\quad \left(t,S\right) \mapsto e^{tM}S.
\end{align*}
The flow exists for all $t\in \Rfield$ and has the flow properties $\varphi(0,S)=S$ and $\varphi(t_2,\varphi(S,t_1))=\varphi(t_1+t_2,S)$ for all $S\in \Grm{n}{n+m}$ and $t_1,t_2\in \Rfield$.

In addition it holds that $U(t)$ is nonsingular as long as $W$ exists.
This motivates us to consider the set of all graph subspaces of $\Grm{n}{n+m}$
\begin{align*}
    \GrmG{n}{n+m}:= \left\{ \range\left( \begin{bmatrix} U \\ V \end{bmatrix} \right) \mymid U\in \Rmat{n}{n},~ V\in \Rmat{m}{n},~\det U \neq 0 \right\}\subseteq \Grm{n}{n+m},
\end{align*}
together with the function
\begin{align*}
    \psi: \GrmG{n}{n+m} \to \Rmat{m}{n},\quad \range\left( \begin{bmatrix} U \\ V \end{bmatrix} \right) \mapsto V U^{-1}.
\end{align*}
The function $\psi$ is well defined, as it does not depend on the basis of the graph subspace.
Thus, we have that
\begin{align*}
        W(t) &=
        \psi \left( \range\left( \begin{bmatrix} U(t) \\ V(t) \end{bmatrix}\right) \right) =
        \psi \left(\varphi\left(t, \range\left( \begin{bmatrix} I_n \\ M_0 \end{bmatrix} \right) \right) \right),
\end{align*}
and
\begin{align*}
        \psi^{-1}\left(W(t)\right) &=
        \range\left(\begin{bmatrix} I_n \\  W(t) \end{bmatrix} \right) =
        \range\left(\begin{bmatrix} I_n \\ V(t){U(t)}^{-1} \end{bmatrix} \right) =
        \range\left(\begin{bmatrix} U(t) \\ V(t) \end{bmatrix} \right) \\ &=
        \varphi\left( t,  \range\left(\begin{bmatrix} I_n \\ M_0 \end{bmatrix} \right) \right),
\end{align*}
as long as the solution $W$ exists. Therefore the solution of the \NDRE{}~\eqref{eqn:nonsymdre} induces a flow on the Grassmanian manifold. The solution $W$ can be recovered from the flow by using $\psi$,
and, vice versa, the flow can be obtained from the solution of the \NDRE{}~\eqref{eqn:nonsymdre} using $\psi^{-1}$.

    \subsection{Solution Formulas}\label{sec:radons_lemma:solution_formulas}

Radon's Lemma~(Thm.~\ref{thm:radon}) enables a certain solution representations for the \DRE~\eqref{eqn:ch_adre:auto_dre}:
Theorem~\ref{thm:ch_adre:dre_existence} ensures that the \DRE{}~\eqref{eqn:ch_adre:auto_dre} has a unique solution for $t\geq 0$.
By Radon's Lemma~(Thm.~\ref{thm:radon}) we have that $U(t)$ is nonsingular for all $t \geq 0 $.

Let $H:= \begin{bmatrix} A & -BB^{T} \\ -C^{T}C & -A^T \end{bmatrix} \in \Rmat{2n}{2n}$ be the Hamiltonian matrix corresponding to the \DRE{}~\eqref{eqn:ch_adre:auto_dre}.
The matrices $U(t)$ and $V(t)$ are determined by the linear initial value problem
\begin{align}
    \begin{bmatrix} \dot{U}(t) \\ \dot{V}(t) \end{bmatrix} & = -H
    \begin{bmatrix} U(t) \\ V(t) \end{bmatrix}, \
    \begin{bmatrix}U(0)\\ V(0)\end{bmatrix} = \begin{bmatrix} I_n \\ X_0 \end{bmatrix}. \label{eqn:HamiltonianUV}
\end{align}
We obtain
\begin{align*}
    \begin{bmatrix} U(t) \\ V(t) \end{bmatrix}  = e^{-tH} \begin{bmatrix} I_n \\ X_0 \end{bmatrix}.
\end{align*}
The strategy is to decompose the Hamiltonian matrix $H$, such that~\eqref{eqn:HamiltonianUV} decouples.

\begin{theorem}[Solution representation \rz{I} for \DRE~\eqref{eqn:ch_adre:auto_dre},~{\cite{Rus88}}]\label{thm:first_sol_formula}
    Let $X\in \Rmat{n}{n}$ be any solution of the \ARE{}~\eqref{eqn:ch_adre:are}.
    Then the solution of the \DRE~\eqref{eqn:ch_adre:auto_dre} for $t \geq 0$  is given by
    \begin{align*}
        X(t)      & \phantom{:}= X - e^{t{(A-BB^T X^T)}^T}\tilde{X} {\left( I_n - \int\limits_{0}^{t} e^{s(A-BB^T X)} BB ^T e^{s{(A-BB^T X^T)}^T}\dx{s} \tilde{X} \right)}^{-1} e^{t(A-BB^T X)}, \\
        \tilde{X} & := X-X_0.
     \end{align*}
\end{theorem}

\begin{proof}
    We use $T:=\begin{bmatrix} I_n & 0 \\ X & I_n \end{bmatrix}$ and apply a similarity transformation to $H$,
    \begin{align*}
        T^{-1} H T
        &=
        \begin{bmatrix} I_n & 0 \\ -X & I_n \end{bmatrix} \begin{bmatrix} A & -BB^{T} \\ -C^{T}C & -A^T \end{bmatrix} \begin{bmatrix} I_n & 0 \\ X & I_n \end{bmatrix}
        \\ &=
        \begin{bmatrix} A-B B^{T} X & -BB^T \\ 0 & -{(A-BB^T X^T)}^T \end{bmatrix}
        =:\tilde{H}.
    \end{align*}
    This gives
    \begin{align*}
        \begin{bmatrix} U(t) \\ V(t) \end{bmatrix}                       =
        e^{-tH}\begin{bmatrix} I_n \\ X_0\end{bmatrix}                   =
        e^{-tT \tilde{H}T^{-1}} \begin{bmatrix}I_n \\ X_0\end{bmatrix}   =
        T e^{-t \tilde{H}}T^{-1} \begin{bmatrix}I_n \\ X_0\end{bmatrix}  =
        T e^{-t \tilde{H}} \begin{bmatrix}I_n \\ X_0-X\end{bmatrix}      =:
        T \begin{bmatrix}\tilde{U}(t) \\ \tilde{V}(t) \end{bmatrix}.
    \end{align*}
    Clearly $\tilde{U}$ and $\tilde{V}$ are determined by the solution of the initial value problem
    \begin{align*}
        \begin{bmatrix}\dot{\tilde{U}}(t) \\ \dot{\tilde{V}}(t) \end{bmatrix} & =
        -\tilde{H}
        \begin{bmatrix} \tilde{U}(t) \\ \tilde{V}(t) \end{bmatrix}=
        \begin{bmatrix} -(A-BB^T X) & BB^{T} \\ 0 & {(A-BB^T X^T)}^{T}  \end{bmatrix}
        \begin{bmatrix} \tilde{U}(t) \\ \tilde{V}(t) \end{bmatrix},\
        \begin{bmatrix} \tilde{U}(0) \\ \tilde{V}(0) \end{bmatrix} = \begin{bmatrix} I_n \\ X_0 -X \end{bmatrix}.
        \intertext{By using the variation of constants formula~\cite[\S 18]{Wal98} we obtain that  $\tilde{U}$ and $\tilde{V}$ are given by}
        \tilde{V}(t)               & = -e^{t{(A-B B^T X^T)}^T}(X-X_0),                                                                                     \\
        \tilde{U}(t)               & = e^{-t{(A-B B^T X)}} + \int\limits_{0}^{t}  e^{-(t-s)(A-BB^T X)} BB^T \tilde{V}(s) \dx{s}                          \\
                                   & = e^{-t(A-BB^T X)} \left( I_n - \int\limits_{0}^{t} e^{s(A-BB^T X)} BB ^T e^{s{(A-BB^T X^T)}^T}\dx{s} (X-X_0)\right).
    \end{align*}
    Since $\tilde{U}(t)=U(t)$ is nonsingular for all $t \geq 0$ and the matrix exponential is nonsingular, the matrix in brackets is also nonsingular for all $t\geq 0$.
    Finally we obtain
    \begin{align*}
        V(t) & = X\tilde{U}(t) + \tilde{V}(t),                            \\
        X(t) & = V(t) {U(t)}^{-1} = X + \tilde{V}(t) {\tilde{U}(t)}^{-1}.
    \end{align*}
\end{proof}

The formula was presented in~\cite{Rus88} without proof.
Since the existence of the involved inverse is not trivially established, we provide a proof.

\begin{theorem}[Solution Representation \rz{II} for \DRE~\eqref{eqn:ch_adre:auto_dre},{~\cite[Thm. 1]{CalWW94}},{~\cite{Rad11}}]\label{thm:second_sol_formula}
    Let $(A,B)$ be stabilizable and $(A,C)$ be detectable and $X_{\infty}\in \Rmat{n}{n}$ be the unique symmetric positive definite stabilizing solution of the \ARE~\eqref{eqn:ch_adre:are}.
    Moreover let $\hat{A}:= A- BB^ T X_{\infty}$ and $X_L\in \Rmat{n}{n}$ be the unique symmetric positive semidefinite solution of the Lyapunov equation
    \begin{align}
        \hat{A} X_L + X_L \hat{A}^ T + BB^T = 0. \label{eqn:closed_loop_lyap}
    \end{align}
    Then the solution of the \DRE~\eqref{eqn:ch_adre:auto_dre} for $t\geq 0$ is given by
    \begin{align*}
        X(t) =  X_{\infty}  - e^{t{\hat{A}}^T}(X_{\infty}-X_0){\left( I_n - (X_L  - e^{t \hat{A}} X_L e^{t {\hat{A}}^T})(X_{\infty}-X_0)\right)}^{-1} e^{t \hat{A}}.
    \end{align*}
\end{theorem}
\begin{proof}
    Similar to the proof of Theorem~\ref{thm:first_sol_formula} we use similarity transformations to decompose the Hamiltonian matrix $H$.
    This is also known as a Riccati-Lyapunov transformation~\cite[Ch. 3.1.1.]{AboFIJ03}.
    We obtain
    \begin{align*}
        T                                           :       &= \begin{bmatrix}   I_n         & 0         \\ X_{\infty}   & I_n           \end{bmatrix},\
        T^{-1} H T                         \phantom{:}      = \begin{bmatrix}   \hat{A}     & -BB^T     \\ 0            & -\hat{A}^T    \end{bmatrix}   =: \tilde{H}, \\
        \tilde{T}                                   : &     = \begin{bmatrix}   I_n         & -X_L      \\ 0            & I_n           \end{bmatrix},\
        \tilde{T}^{-1} \tilde{H} \tilde{T} \phantom{:}      = \begin{bmatrix}   \hat{A}     & 0         \\ 0            & -\hat{A}^T    \end{bmatrix}   =: \hat{H}.
    \end{align*}
    We thus get
    \begin{align}
        \begin{bmatrix} U(t) \\ V(t) \end{bmatrix}
         & =
         e^{-t H }
         \begin{bmatrix}  I_n \\ X_0 \end{bmatrix}
         =
         e^{-{t (T \tilde{T})}
             \hat{H}
         {(T \tilde{T})}^{-1}}
         \begin{bmatrix} I_n \\ X_0 \end{bmatrix}
         =
         (T \tilde{T}) e^{-t\hat{H}}{(T \tilde{T})}^{-1} \begin{bmatrix} I_n \\ X_0 \end{bmatrix}     \notag \\
         & =
        \begin{bmatrix} I_n                   & -X_L  \\ X_{\infty}   &  I_n - X_{\infty} X_L     \end{bmatrix}
        \begin{bmatrix} e^{-t \hat{A}}        & 0     \\ 0            & e^{t\hat{A}^T}            \end{bmatrix}
        \begin{bmatrix} I_n - X_L X_{\infty}  &  X_L  \\ -X_{\infty}  & I_n                       \end{bmatrix}
        \begin{bmatrix} I_n \\ X_0 \end{bmatrix}                                                    \notag  \\
                                              & =
                                              \begin{bmatrix}
                                                  e^{-t \hat{A}}\left(I_n - X_L X_{\infty} \right) + X_L e^{t \hat{A}} X_{\infty}
             &
             e^{-t \hat{A}}X_L - X_L e^{t \hat{A}^T}
             \\
             X_{\infty}e^{-t \hat{A}}\left( I_n - X_L X_{\infty} \right) - \left(I_n -X_{\infty} X_L  \right) e^{t \hat{A}}X_{\infty}
             &
             X_{\infty}e^{-t \hat{A}} X_L + \left(I_n - X_{\infty} X_L \right)e^{t \hat{A}}
                                              \end{bmatrix}
                                              \begin{bmatrix} I_n \\ X_0 \end{bmatrix}                                                    \notag  \\
                                              & =
                                              \begin{bmatrix}
                                                  e^{-t\hat{A}}(I_n - X_L(X_{\infty}-X_0)) + X_L e^{t\hat{A}^T} (X_{\infty}-X_0)
                                                  \\
                                                  X_{\infty} e^{-t\hat{A}}(I_n - X_L(X_{\infty}-X_0)) - (X_{\infty}X_L + I_n)e^{t\hat{A}^T}(X_{\infty}-X_0)
                                              \end{bmatrix}. \label{thm:second_sol_formula:hamiltonian}
    \end{align}
    Now observe that
    \begin{align*}
        U(t) & = e^{-t \hat{A}}\left( I_n - (X_L - e^{t \hat{A}} X_L e^{t \hat{A}^T})(X_{\infty} - X_0)\right),                                                                            \\
        V(t) & = X_{\infty} e^{-t\hat{A}}\left( I_n - (X_L  - e^{t \hat{A}} X_L e^{t \hat{A}^T})(X_{\infty} - X_0)\right) - e^{t\hat{A}^T}(X_{\infty}-X_{0})                               \\
             & = X_{\infty} U(t) - e^{t\hat{A}^T}(X_{\infty}-X_{0}),
             \intertext{therefore}
        X(t) & = V(t){U(t)}^{-1} = X_{\infty}  - e^{t{\hat{A}}^T}(X_{\infty}-X_{0}){\left( I_n - \left(X_L - e^{t \hat{A}} X_L e^{t {\hat{A}}^T} \right)(X_{\infty} - X_0)\right)}^{-1} e^{t \hat{A}}.
    \end{align*}
\end{proof}

In~\cite[Ch. 15.4]{AndM71} one can find another solution formula, which holds under more restrictive assumptions.
A solution formula based on the Jordan canonical form is given in~\cite[Thm. 3.2.1]{AboFIJ03}.

    \subsection{Davison-Maki Methods}\label{sec:radons_lemma:davison_maki_methods}


The \emph{Davison-Maki method} for the \NDRE{}~\eqref{eqn:nonsymdre} was proposed in~\cite{DavM73}.
The method is based on first computing the matrix exponential $e^{hM}$ for a given step size $h>0$.
According to Radon's Lemma~(Thm.~\ref{thm:radon}) we have that
\begin{align*}
    \begin{bmatrix} U(h) \\ V(h) \end{bmatrix} = e^{hM} \begin{bmatrix} I_n \\ M_0 \end{bmatrix},\quad W(h) = V(h){U(h)}^{-1}.
\end{align*}
The next step is then to make use of the semigroup property of the matrix exponential
\begin{align*}
    \begin{bmatrix} U(2h) \\ V(2h) \end{bmatrix} = e^{2hM} \begin{bmatrix} I_n \\ M_0 \end{bmatrix} = {\left(e^{hM}\right)}^{2} \begin{bmatrix} I_n \\ M_0 \end{bmatrix},\quad W(2h) = V(2h) {U(2h)}^{-1}.
\end{align*}
For the further steps we obtain
\begin{align}
    \begin{bmatrix} U(kh) \\ V(kh) \end{bmatrix} =  {\left(e^{hM}\right)}^{k} \begin{bmatrix} I_n \\ M_0 \end{bmatrix},\quad W(kh) = V(kh) {U(kh)}^{-1}. \label{eqn:alg:nonsymdre:davison_maki:variant1}
\end{align}
Another variant of the \emph{Davison-Maki method} updates $U$ and $V$ instead of the matrix exponential.
The variant follows from
\begin{align}
    \begin{bmatrix} U(kh) \\ V(kh) \end{bmatrix} =
    e^{khM}\begin{bmatrix} I_n \\ M_0 \end{bmatrix} =
    e^{hM} e^{(k-1)h M } \begin{bmatrix} I_n \\ M_0 \end{bmatrix} =
    e^{hM} \begin{bmatrix} U((k-1)h) \\ V((k-1)h) \end{bmatrix}. \label{eqn:alg:nonsymdre:davison_maki:variant2}
\end{align}
Both variants of the method are given in Algorithm~\ref{alg:nonsymdre:davison_maki}.

\begin{algorithm}[H]
    \caption{Davison-Maki method for the \NDRE~\eqref{eqn:nonsymdre}~\cite{DavM73,KenL85}}\label{alg:nonsymdre:davison_maki}
    \begin{algorithmic}[1]
        \Assumption{The \NDRE~\eqref{eqn:nonsymdre} has a solution $W\colon \interor{0}{t_f}\rightarrow \Rmat{m}{n}$.}
        \Require{Real matrices $M_0$ and $M_{ij}$ as in Theorem~\ref{thm:radon}, step size $h>0$ and final time $t_f>0$.}
        \Ensure{Matrices $W_{k}$, such that $W(kh)=W_{k}$ for $k\in \Nzero$ and $kh<t_f$.}
        \State{$W_{0} = M_0$;}
        \State{$k = 1$;}
        \Statex{\% Compute matrix exponential e.g.\ by a scaling and squaring method:}
        \State{$\Theta_h=\exp\left({h\begin{bmatrix}M_{11} & M_{12} \\ M_{21} & M_{22}\end{bmatrix}}\right)$;}
        \vspacealgorithm{}
        \Statex{\texttt{Variant with matrix exponential update:}}
        \State{$\Theta = \Theta_h$;}
        \While{$kh<t_f$}\label{alg:nonsymdre:davison_maki:time_stepping1_start}
        \State{Partition
            $
            \begin{blockarray}{ccc}
                 & n & m \\
                \begin{block}{c[cc]}
                    n & \Theta_{11} & \Theta_{12} \\
                    m & \Theta_{12} & \Theta_{22} \\
                \end{block}
            \end{blockarray}
            = \Theta;
            $
        }
        \State{$U_{\texttt{dm}} = \Theta_{11} + \Theta_{12} M_{0}$;}
        \State{$V_{\texttt{dm}} = \Theta_{21} + \Theta_{22} M_{0}$;}
        \State{$W_k=V_{\texttt{dm}}{U_{\texttt{dm}}}^{-1}$;}
        \State{$\Theta = \Theta \Theta_h$;}
        \State{$k=k+1$;}
        \EndWhile{}\label{alg:nonsymdre:davison_maki:time_stepping1_stop}
        \vspacealgorithm{}
        \Statex{\texttt{Variant with updating $U$ and $V$:}}
        \State{$U_{\texttt{dm}}=I_n$;}
        \State{$V_{\texttt{dm}}=M_0$;}
        \State{Partition
            $
            \begin{blockarray}{ccc}
                 & n & m \\
                \begin{block}{c[cc]}
                    n & \Theta_{11} & \Theta_{12} \\
                    m & \Theta_{12} & \Theta_{22} \\
                \end{block}
            \end{blockarray}
            = \Theta;
            $
        }
        \While{$kh<t_f$}\label{alg:nonsymdre:davison_maki:time_stepping2_start}
        \State{$U_{\texttt{dm}} = \Theta_{11}U_{\texttt{dm}} + \Theta_{12} V_{\texttt{dm}}$;}
        \State{$V_{\texttt{dm}} = \Theta_{21}U_{\texttt{dm}} + \Theta_{22} V_{\texttt{dm}}$;}
        \State{$W_k= V_{\texttt{dm}} U_{\texttt{dm}}^{-1}$;}
        \State{$k=k+1$;}
        \EndWhile{}\label{alg:nonsymdre:davison_maki:time_stepping2_stop}
    \end{algorithmic}
\end{algorithm}

When the \emph{Davison-Maki method} (Alg.~\ref{alg:nonsymdre:davison_maki}) is applied to the \DRE{}~\eqref{eqn:ch_adre:auto_dre}, usually numerical instabilities occur
which are due to the fact that each block
of $e^{-tH}$ as well as $U(t)$ and $V(t)$ contains the matrix $e^{-t\hat{A}}$, cp.\ equation~\eqref{thm:second_sol_formula:hamiltonian}.
Since $\hat{A}=A-BB^T X_{\infty}$ is stable, the matrix exponential of $-t\hat{A}$  exhibits exponential growth which becomes problematic for large $t$.
The occurrence of these numerical problems with the \emph{Davison-Maki method} (Alg.~\ref{alg:nonsymdre:davison_maki}) was also pointed out in~\cite{117906,Vau69,Lau82,KenL85}.
Another reason is that the spectrum of a real Hamiltonian matrix comes in quadruples, that is $\spektrum{H} = \{\lambda_1,\ldots,\lambda_n,-\lambda_1,\ldots,-\lambda_n\}$ with $\Repart(\lambda_i)\leq 0$.
Therefore, usually, the spectrum of the Hamiltonian contains eigenvalues with positive real part and, thus, also it's matrix exponential grows~\cite[Prop. 2.3.1]{MeyO17}.

A suitable modification of the \emph{Davison-Maki method} (Alg.~\ref{alg:nonsymdre:davison_maki}) was proposed in~\cite{KenL85}, but the modified method originates back to~\cite[p.~9]{KalE66}.
By Radon's Lemma~(Thm.~\ref{thm:radon}), as laid out in Section~\ref{sec:radons_lemma:flow_on_the_grassmanian_manifolds}, we have the identity
\begin{align*}
    W(kh)
    &=
    \psi\left(\range\left(\begin{bmatrix} U(kh)\\ V(kh) \end{bmatrix}\right)\right) =
    \psi\left( e^{kh M} \range\left(\begin{bmatrix} I_n\\ M_0 \end{bmatrix}\right) \right) =
    \psi\left( e^{h M}  \range\left(\begin{bmatrix} U((k-1)h)\\ V((k-1)h) \end{bmatrix}\right) \right) \\
    &=
    \psi\left( e^{h M}  \range\left(\begin{bmatrix} I_n\\ W((k-1)h) \end{bmatrix}\right) \right)
    =
    \psi\left(\range\left(e^{h M} \begin{bmatrix} I_n\\ W((k-1)h) \end{bmatrix}\right) \right).
\end{align*}
Therefore the iteration for the \emph{modified Davison-Maki method} is given by
\begin{align}
    \begin{bmatrix} \tilde{U} \\ \tilde{V} \end{bmatrix} := e^{h M}\begin{bmatrix} I_n \\ W((k-1)h) \end{bmatrix},\quad W(kh)= \tilde{V}  \tilde{U}^{-1}. \label{eqn:alg:nonsymdre:mod_davison_maki}
\end{align}
The \emph{modified Davison-Maki method} is given in Algorithm~\ref{alg:nonsymdre:mod_davison_maki}.
\begin{algorithm}[H]
    \caption{Modified Davison-Maki method for the~\NDRE~\eqref{eqn:nonsymdre}~\cite{KalE66,KenL85}}\label{alg:nonsymdre:mod_davison_maki}
    \begin{algorithmic}[1]
        \Assumption{The \NDRE~\eqref{eqn:nonsymdre} has a solution $W\colon \interor{0}{t_f}\rightarrow \Rmat{m}{n}$.}
        \Require{Real matrices $M_0$ and $M_{ij}$ as in Theorem~\ref{thm:radon}, step size $h>0$, final time $t_f>0$ and a moderate large number $tol_{\texttt{exp}}>0$.}
        \Ensure{Matrices $W_{k}$, such that $W(kh)=W_{k}$ for $k\in \Nzero$ and $kh<t_f$.}
        \State{$W_{0} = M_0$;}
        \State{$k=1$;}
        \Statex{\% Compute matrix exponential e.g.\ by a scaling and squaring method:}
        \State{$\Theta=\exp\left({h\begin{bmatrix}M_{11} & M_{12} \\ M_{21} & M_{22}\end{bmatrix}}\right)$;}\label{alg:nonsymdre:mod_davison_maki:expm}
        \vspacealgorithm{}
        \Statex{\% Check the norm  of the matrix exponential:}
        \If{$\norm{\Theta}_1 > tol_{\texttt{exp}}$}\label{alg:nonsymdre:mod_davison_maki:norm_matrix_exponential}
        \State{return Error(\glqq1-Norm of the matrix exponential is too large, decrease the step size $h$\grqq)}.
        \EndIf{}
        \State{Partition
            $
            \begin{blockarray}{ccc}
                 & n & m \\
                \begin{block}{c[cc]}
                    n & \Theta_{11} & \Theta_{12} \\
                    m & \Theta_{12} & \Theta_{22} \\
                \end{block}
            \end{blockarray}
            = \Theta;
            $
        }\label{alg:nonsymdre:mod_davison_maki:partition}
        \While{$kh<t_f$}\label{alg:nonsymdre:mod_davison_maki:time_stepping_start}
        \State{$U_{\texttt{mod\_dm}} = \Theta_{11} + \Theta_{12} W_{k-1}$;}
        \State{$V_{\texttt{mod\_dm}} = \Theta_{21} + \Theta_{22} W_{k-1}$;}
        \State{$W_{k} =  V_{\texttt{mod\_dm}}  U_{\texttt{mod\_dm}}^{-1}$;}\label{alg:nonsymdre:mod_davison_maki:transform}
        \State{$k = k+1$;}
        \EndWhile{}\label{alg:nonsymdre:mod_davison_maki:time_stepping_stop}
    \end{algorithmic}
\end{algorithm}
A decrease of the step size $h>0$, does not improve the accuracy in general, because the iteration is exact.
The accuracy is determined by the accuracy of the matrix exponential computation and the matrix inversion.
The step size cannot be chosen arbitrary large as the matrix exponential may become too large in norm.
In practice we suggest to compute the norm of the matrix exponential before the iteration starts. If the norm is too large, then the step size has to be decreased.
In the $k$-th iteration of Algorithm~\ref{alg:nonsymdre:mod_davison_maki} we have
\begin{align*}
    \begin{bmatrix} U_{\texttt{mod\_dm}} \\ V_{\texttt{mod\_dm}} \end{bmatrix}
    &=
    e^{hM}
    \begin{bmatrix} I_n \\ W((k-1)h) \end{bmatrix}
    =
    \begin{bmatrix}
        \Theta_{11} & \Theta_{12} \\
        \Theta_{21} & \Theta_{22}
    \end{bmatrix}
    \begin{bmatrix} I_n \\ W((k-1)h) \end{bmatrix}, \\
    \intertext{and the norm of the iterates can be bounded by}
    \norm{U_{\texttt{mod\_dm}}} & \leq \norm{\Theta_{11}} + \norm{\Theta_{12}} \norm{ W((k-1)h)},    \\
    \norm{V_{\texttt{mod\_dm}}} & \leq \norm{\Theta_{21}} + \norm{\Theta_{22}} \norm{ W((k-1)h)}.
\end{align*}
For small step sizes of $h>0$ it holds  $e^{hM} \approx I_{n+m} + hM$ and $\Theta_{11} \approx I_n + h M_{11}$, $\Theta_{12} \approx hM_{12}$, $\Theta_{21} \approx hM_{21}$ and $\Theta_{22} \approx I_m + hM_{22}$.
Therefore for small enough step size and moderate norm of the solution $\norm{W(t)}$, the norm of the iterates cannot grow heavily in contrast to Algorithm~\ref{alg:nonsymdre:davison_maki}.
If the norm of the iterates becomes too large during iteration, the step size should be decreased. Assume that the matrix exponential in line~\ref{alg:nonsymdre:mod_davison_maki:expm} of Algorithm~\ref{alg:nonsymdre:mod_davison_maki}
was approximated by using the scaling and squaring method, then the intermediates of the squaring phase can be used and the matrix exponential needs not be recomputed from scratch.

\begin{example}[Exponential Growth Davison-Maki method]\label{ex:exponential_growth_davison_maki}
We applied the Davison-Maki method \textup{(Alg.~\ref{alg:nonsymdre:davison_maki})} with step size $h=2^{-8}$ to a \DRE{} with the same matrices $A,B,C$ and $X_0$ as for \textup{Example~\ref{ex:eigenvalue_decay}}.
We plot the $2$-norm of the iterates $U_{\textup{\texttt{dm}}}$ and $V_{\textup{\texttt{dm}}}$ as well as the $2$-norm condition number of $U_{\textup{\texttt{dm}}}$ on the interval $[0,1]$.
The plot shows that all quantities grow exponentially over time.
Therefore, eventually, either a floating point overflow will occur or the matrix inversion ceases to be executed accurately.
\textup{Figure~\ref{fig:ch_solution_formula:mod_davison_maki_fail}} shows the same quantities for the iterates $U_{\textup{\texttt{mod\_dm}}}$ and $V_{\textup{\texttt{mod\_dm}}}$
of the modified Davison-Maki method \textup{(Alg.~\ref{alg:nonsymdre:mod_davison_maki})}.
\begin{figure}[H]
    \begin{subfigure}[b]{0.30\textwidth}
        \centering
        \resizebox{\linewidth}{!}{
            \includegraphics{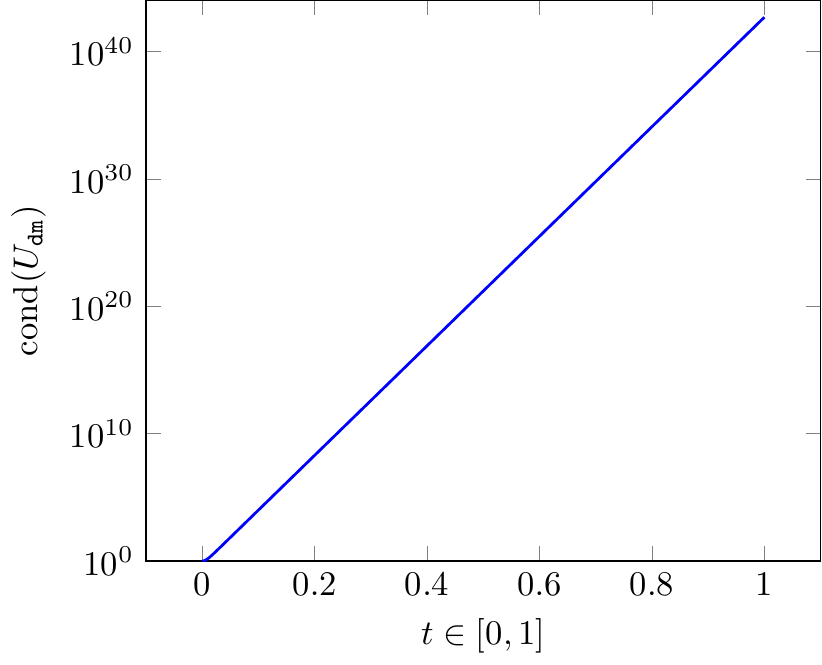}
        }
        \caption{Condition number of $U_{\texttt{dm}}$.}\label{fig:ch_solution_formula:davison_maki_fail:condU}
    \end{subfigure}
    \begin{subfigure}[b]{0.30\textwidth}
        \centering
        \resizebox{\linewidth}{!}{
            \includegraphics{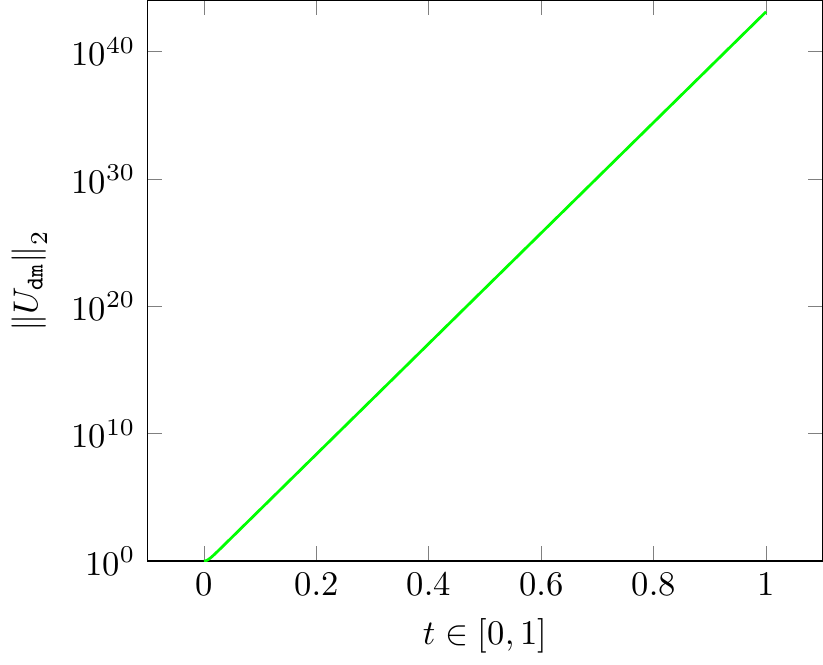}
        }
        \caption{$2$-norm $U_{\texttt{dm}}$.}\label{fig:ch_solution_formula:davison_maki_fail:normU}
    \end{subfigure}
    \begin{subfigure}[b]{0.30\textwidth}
        \centering
        \resizebox{\linewidth}{!}{
            \includegraphics{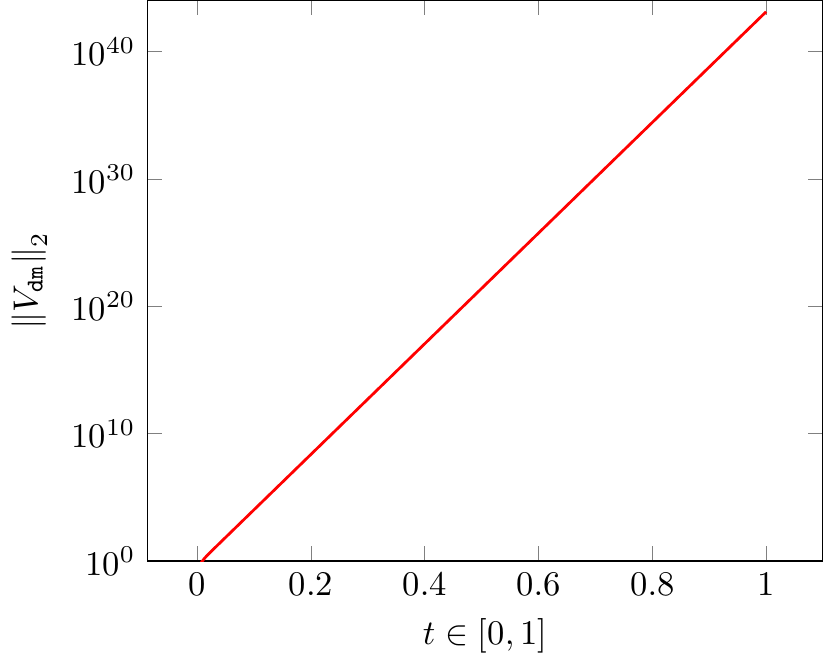}
        }
        \caption{$2$-norm $V_{\texttt{dm}}$.}\label{fig:ch_solution_formula:davison_maki_fail:normV}
    \end{subfigure}
    \caption{Davison-Maki method Algorithm~\ref{alg:nonsymdre:davison_maki} and the growth of $U_{\texttt{dm}}$ and $V_{\texttt{dm}}$.}\label{fig:ch_solution_formula:davison_maki_fail}
\end{figure}

\begin{figure}[H]
    \begin{subfigure}[b]{0.30\textwidth}
        \centering
        \resizebox{\linewidth}{!}{
            \includegraphics{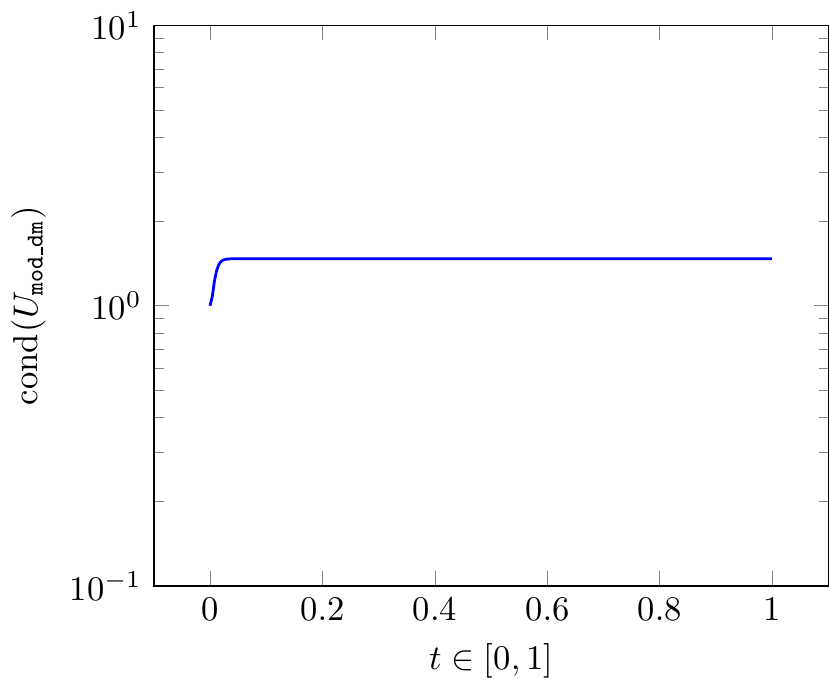}
        }
        \caption{Condition number of $U_{\texttt{mod\_dm}}$.}\label{fig:ch_solution_formula:mod_davison_maki_fail:condU}
    \end{subfigure}
    \begin{subfigure}[b]{0.30\textwidth}
        \centering
        \resizebox{\linewidth}{!}{
            \includegraphics{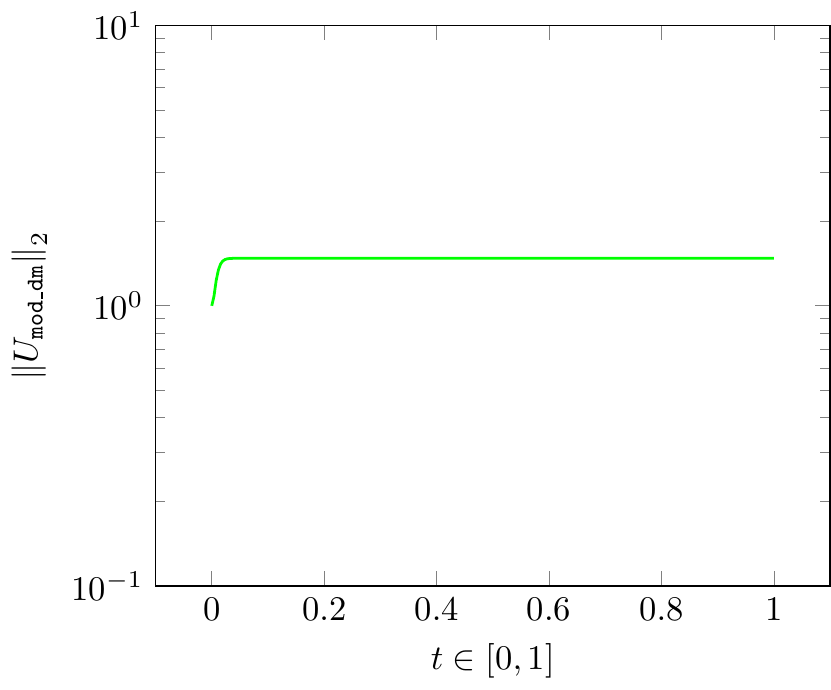}
        }
        \caption{$2$-norm $U_{\texttt{mod\_dm}}$.}\label{fig:ch_solution_formula:mod_davison_maki_fail:normU}
    \end{subfigure}
    \begin{subfigure}[b]{0.30\textwidth}
        \centering
        \resizebox{\linewidth}{!}{
            \includegraphics{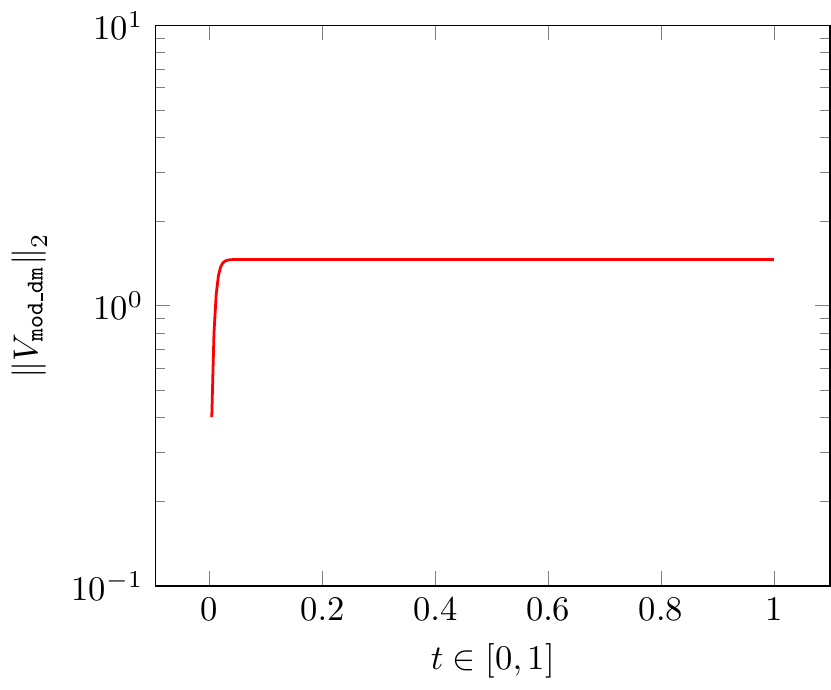}
        }
        \caption{$2$-norm $V_{\texttt{mod\_dm}}$.}\label{fig:ch_solution_formula:mod_davison_maki_fail:normV}
    \end{subfigure}
    \caption{Modified Davison-Maki method Algorithm~\ref{alg:nonsymdre:mod_davison_maki} and the growth of $U_{\texttt{mod\_dm}}$ and $V_{\texttt{mod\_dm}}$.}\label{fig:ch_solution_formula:mod_davison_maki_fail}
\end{figure}
\end{example}

If a symmetric solution is expected, then line~\ref{alg:nonsymdre:mod_davison_maki:transform} in Algorithm~\ref{alg:nonsymdre:mod_davison_maki} should be altered with $W_k = \frac{1}{2}\left( W_k + W_k^{T} \right)$,
because due to numerical errors the symmetry will be lost after some iterations.

Any computational efficient norm can also be used for the matrix exponential in Algorithm~\ref{alg:nonsymdre:mod_davison_maki} line~\ref{alg:nonsymdre:mod_davison_maki:norm_matrix_exponential}.
The \emph{modified Davison-Maki method} is also more efficient than the \emph{Davison-Maki method} in both variants,
because less matrix-matrix products are needed by time step,
compare Algorithm~\ref{alg:nonsymdre:mod_davison_maki} line~\ref{alg:nonsymdre:mod_davison_maki:time_stepping_start}-\ref{alg:nonsymdre:mod_davison_maki:time_stepping_stop}
with Algorithm~\ref{alg:nonsymdre:davison_maki} line~\ref{alg:nonsymdre:davison_maki:time_stepping1_start}-\ref{alg:nonsymdre:davison_maki:time_stepping1_stop} and
line~\ref{alg:nonsymdre:davison_maki:time_stepping2_start}-\ref{alg:nonsymdre:davison_maki:time_stepping2_stop}.

The computational cost apart from matrix exponential computation grows linearly with the time step size $h$, compare Algorithm~\ref{alg:nonsymdre:mod_davison_maki} line~\ref{alg:nonsymdre:mod_davison_maki:time_stepping_start}-\ref{alg:nonsymdre:mod_davison_maki:time_stepping_stop}.

The intermediates $U_{\texttt{dm}},  V_{\texttt{dm}}$ from Algorithm~\ref{alg:nonsymdre:mod_davison_maki} and $U_{\texttt{mod\_dm}},  V_{\texttt{mod\_dm}}$ from Algorithm~\ref{alg:nonsymdre:davison_maki}
are usually different. The next lemma shows the connection.

\begin{lemma}
    In the $k$-th iteration of \textup{Algorithm~\ref{alg:nonsymdre:davison_maki}} and \textup{Algorithm~\ref{alg:nonsymdre:mod_davison_maki}}, the iterates
    $U_{\textup{\texttt{dm}}},  V_{\textup{\texttt{dm}}}$
    and
    $U_{\textup{\texttt{mod\_dm}}},  V_{\textup{\texttt{mod\_dm}}}$
    fulfill
\begin{align*}
    \range\left( \begin{bmatrix}  U_{\textup{\texttt{dm}}} \\  V_{\textup{\texttt{dm}}} \end{bmatrix} \right)
    =
    \range\left( \begin{bmatrix}  U_{\textup{\texttt{mod\_dm}}} \\  V_{\textup{\texttt{mod\_dm}}} \end{bmatrix} \right).
\end{align*}
\end{lemma}
\begin{proof}
From equation~\eqref{eqn:alg:nonsymdre:davison_maki:variant1}  it follows
\begin{align*}
    \range\left( \begin{bmatrix}  U_{\textup{\texttt{dm}}} \\  V_{\textup{\texttt{dm}}} \end{bmatrix} \right) =
    \range\left( e^{hkM} \begin{bmatrix}  I_n \\  M_0 \end{bmatrix} \right).
\end{align*}
Equation~\eqref{eqn:alg:nonsymdre:mod_davison_maki} gives
\begin{align*}
    \range\left( \begin{bmatrix}  U_{\textup{\texttt{mod\_dm}}} \\  V_{\textup{\texttt{mod\_dm}}} \end{bmatrix} \right)
    &=
    \range\left( e^{hM} \begin{bmatrix}  I_n        \\  W((k-1)h) \end{bmatrix} \right)
    =
    \range\left( e^{hM} \begin{bmatrix}  I_n        \\ V((k-1)h) {U((k-1)h)}^{-1} \end{bmatrix} \right) \\
    &=
    \range\left( e^{hM} \begin{bmatrix}  U((k-1)h)  \\ V((k-1)h)  \end{bmatrix} \right)
    =
    \range\left( e^{khM} \begin{bmatrix}  I_n  \\ M_0  \end{bmatrix} \right).
\end{align*}

\end{proof}

\section{Galerkin Approach for Large-Scale Differential Riccati Equations}\label{sec:galerkin_approach_for_large_scale_differential_riccati_equations}

In this section we develop a feasible numerical approach for large-scale differential Riccati equations.
We consider the \DRE~\eqref{eqn:ch_adre:auto_dre} and assume that $X_0=0$.
We develop the Galerkin approach based on two theoretical considerations.
First we use the solution formula of Theorem~\ref{thm:second_sol_formula}.
We show that the range of the solution $X_{\infty}$ of the \ARE{} is invariant under the action of the closed-loop matrix $A-B{B^T}X_{\infty}$.
It follows then that the action of the matrix exponential of the closed-loop matrix on $X_{\infty}$ has the same property.
This makes the approach consistent in the sense that the evolution does not leave the ansatz space and provides reasoning that the consistency error made by a numerical approximation to these subspaces can be made arbitrarily small.
Moreover, this invariance property allows for a straight-forward low-dimensional approximation of the matrix exponential.
After that we show that, for our proposed choice of a Galerkin basis, a quick decay of the eigenvalues of the solution of the \ARE{} implies a decent approximation of the solution $X(t)$ of the \DRE{}.

The result is a low-dimensional solution space with an accessible formula for the relevant matrix exponential so that we can use the \emph{modified Davison-Maki} (Algorithm~\ref{alg:nonsymdre:mod_davison_maki}) for an efficient solution of the projected Galerkin system.

\subsection{Invariant Subspaces for the Galerkin Approach}\label{sec:galerkin_approach_for_large_scale_differential_riccati_equations:galerkin_approach_using_solution_representations}

First we prove that the range space of the solution $X_{\infty}$ of the~\ARE{} is invariant under the action of the transposed closed-loop matrix ${\left(A-BB^T X_{\infty}\right)}^T$.
\begin{lemma}\label{thm:ch_adre:are_invariance_range}
Let $(A,B)$ be stabilizable, $(A,C)$ be detectable and $X_{\infty} \in \Rmat{n}{n}$ be the unique stabilizing solution of the \ARE{}~\eqref{eqn:ch_adre:are}.
Then $\range(X_{\infty})$ is ${(A-B B^T X_{\infty})}^T$--invariant.
\end{lemma}
\begin{proof}
    We can assume that $X_{\infty}\neq 0$.
    Let the columns of $Q_{\infty} \in \Rmat{n}{p}$ be an orthonormal basis for $\range(X_{\infty})$.
    Then $Q_{\infty}Q_{\infty}^T$ is the orthogonal projection onto $\range(X_{\infty})$.
    We obtain
    \begin{align*}
        Q_{\infty}Q_{\infty}^T X_{\infty} = X_{\infty}.
    \end{align*}
    By Theorem~\ref{thm:ch_adre:are_range}, the columns of $Q_{\infty}$ are also an orthonormal basis for $\mathcal{K}\left(A^T, C^T\right)$.
    The space $\mathcal{K}\left(A^T, C^T\right)$ is $A^T$--invariant. We obtain
    \begin{align*}
        A^T Q_{\infty} = Q_{\infty}Q_{\infty}^T A^T Q_{\infty}.
    \end{align*}
    Finally, we have
    \begin{align*}
        {(A - BB^T X_{\infty})}^T Q_{\infty}
        &=
        Q_{\infty} Q_{\infty}^T A^T Q_{\infty} - Q_{\infty} Q_{\infty}^T X_{\infty} B B ^T Q_{\infty} \\
        &=
        Q_{\infty} \left( Q_{\infty}^T A^T Q_{\infty}  - Q_{\infty}^T X_{\infty} B  B^T Q_{\infty}\right).
    \end{align*}
    This means $\range(X_{\infty})$ is ${(A-B B^T X_{\infty})}^T$--invariant.
\end{proof}

According to Theorem~\ref{thm:second_sol_formula} the solution of the \DRE~\eqref{eqn:ch_adre:auto_dre} is for $t\geq 0$ given by
\begin{align*}
    X(t) &=\phantom{:} X_{\infty}  - e^{t \hat{A}^T} X_{\infty} {\left( I_n - \left( X_L - e^{t\hat{A}}X_L e^{t\hat{A}^T} \right)X_{\infty} \right)}^{-1} e^{t\hat{A}},
\end{align*}
where $\hat{A} = A-BB^T X_{\infty}$.
The identity ${\left( I_n - P(t) \right)}^{-1} = I_n + {\left( I_n - P(t) \right)}^{-1}P(t)$ leads to
\begin{align}
    X(t) &=
    X_{\infty}
         -
    e^{t \hat{A}^T} X_{\infty} e^{t\hat{A}}\notag
            \\
         &\phantom{=}-
    e^{t\hat{A}^T}X_{\infty} {\left( I_n - \left( X_L - e^{t\hat{A}}X_L e^{t\hat{A}^T} \right)X_{\infty}  \right)}^{-1}
    \left( X_L - e^{t\hat{A}}X_L e^{t\hat{A}^T} \right)X_{\infty}   e^{t\hat{A}}.\label{eqn:galerkin:expanded_sol_formula}
\end{align}

\textbf{Derivation by using the exact solution $X_{\infty}$ of the \ARE{}}\\
By Lemma~\ref{thm:ch_adre:are_invariance_range} it holds that $\range\left( X_{\infty} \right)$ is invariant under $\hat{A}^T$.
Assume now that $X_{\infty}$ is given in factorized form, this means that $X_{\infty} = Z_{\infty} Z _{\infty}^T$ and $Z_{\infty}\in \Rmat{n}{p}$ and $1\leq p=\rank\left( X_{\infty} \right)\leq n$.
If $\rank(X_{\infty})=0$, then also $X_{\infty}=0$ as well as the solution $X(t)$.
Now it holds that $ \range(X_{\infty}) = \range(Z_{\infty})$ and consequently $\range(Z_{\infty})$ is invariant under $\hat{A}^T$.

By means of the compact singular value decomposition of $Z_{\infty}$, we obtain matrices $Q_{\infty}\in \Rmat{n}{p}$, $S_{\infty}\in \Rmat{p}{p}$ and $V_{\infty}\in \Rmat{p}{p}$, such that
$ Z_{\infty} = Q_{\infty} S_{\infty} V_{\infty}^T$, $\range\left( Q_{\infty} \right) = \range\left( Z_{\infty}\right)$ and
\begin{align*}
    Z_{\infty} Z_{\infty}^T = Q_{\infty} S_{\infty}^2 Q_{\infty}^T.
\end{align*}
Because of the invariance we get
\begin{align*}
    e^{t\hat{A}^T} Q_{\infty}  = Q_{\infty} e^{t Q_{\infty}^{T} \hat{A}^T Q_{\infty}}.
\end{align*}
Now observe that
\begin{align}
    e^{t\hat{A}^T} X_{\infty}
    =
    e^{t\hat{A}^T}  Z_{\infty} Z_{\infty}^T
    =
    e^{t\hat{A}^T}  Q_{\infty} S_{\infty}^2 Q_{\infty}^T
    =
    Q_{\infty} e^{t Q_{\infty}^{T} \hat{A}^T Q_{\infty}} S_{\infty}^2 Q_{\infty}^T.\label{eqn:galerkin:action_matrix_exponential}
\end{align}
Therefore the solution $X(t)$ can be written in the form
\begin{align}
    X(t) = X_{\infty}  - Q_{\infty}\tilde{X}(t) Q_{\infty}^T. \label{eqn:ch_num_appr:xtilde_dre}
\end{align}
We use the \DRE~\eqref{eqn:ch_adre:auto_dre} and equation~\eqref{eqn:ch_num_appr:xtilde_dre}
and get a differential equation for $\tilde{X}(t)$
\begin{subequations}\label{eqn:ch_num_appr:dre_xtilde_proj}
\begin{align}
    \dot{\tilde{X}}(t)  &= Q_{\infty}^T {\hat{A}}^T Q_{\infty} \tilde{X}(t) + \tilde{X} (t)  Q_{\infty}^T  \hat{A} Q_{\infty}  + \tilde{X} Q_{\infty}^T BB^T Q_{\infty} \tilde{X}(t), \label{eqn:ch_num_appr:dre_xtilde_proj_eqn} \\
    \tilde{X} (0)       &= Q_{\infty}^T X_{\infty}Q_{\infty}. \label{eqn:ch_num_appr:dre_xtilde_proj_initial}
\end{align}
\end{subequations}

\textbf{Derivation by using a low-rank approximation $X_N$ of the exact solution $X_{\infty}$ of the \ARE{}}\\
Let now $Z_N Z_N^T = X_{N} \approx  X_{\infty}$ be a low-rank approximation obtained by a numerical method.
We replace $X_{\infty}$ by $X_N$ in formula~\ref{eqn:galerkin:expanded_sol_formula} and obtain
\begin{align*}
X(t) &\approx           X_N - e^{t{\left( A-BB^T X_N \right)}^T} X_N  e^{t{\left( A-BB^T X_N \right)}}                                                  \\
     &\phantom{\approx}     - e^{t{\left( A-BB^T X_N \right)}^T} X_N {\left( I_n - \left( X_L - e^{t\hat{A}}X_L e^{t\hat{A}^T} \right)X_{\infty}  \right)}^{-1}
                    \left( X_L - e^{t\hat{A}}X_L e^{t\hat{A}^T} \right) X_N e^{t{\left( A-BB^T X_N \right)}}.
\end{align*}

Let $Z_N = Q_N S_N V_N^T$ be the compact singular value decomposition of the low-rank factor.
According to formula~\ref{eqn:galerkin:action_matrix_exponential},
we propose to approximate the action of the matrix exponential by
\begin{align*}
    e^{t{\left( A-BB^T X_N \right)}^T} X_N    &= e^{t{\left( A-BB^T X_N \right)}^T} Z_N Z_N^T = e^{t{\left( A-BB^T X_N \right)}^T} Q_N S_N^2 Q_N^T \\
                                              &\approx Q_N e^{t Q_N^T {\left( A-BB^T X_N \right)}^T Q_N } S_N^2 Q_N^T.
\end{align*}
Therefore we obtain the Galerkin ansatz $X(t)  \approx X_N  - Q_N \tilde{X}_N(t) Q_N^T$ for the numerical approximation.
Again we use the \DRE~\eqref{eqn:ch_adre:auto_dre}
and get a differential equation for $\tilde{X}_N(t)$
\begin{align*}
    \dot{\tilde{X}}_N(t)    &= Q_N^T {\left( A - BB^T X_N \right)}^T Q_N \tilde{X}_N(t) +  \tilde{X}_N(t) Q_N^T {\left( A - BB^T X_N \right)} Q_N       \\
                            &\phantom{=} + \tilde{X}_N(t)Q_N^T B B^T Q_N \tilde{X}_N(t) + Q_N^T\mathcal{R}(X_N)Q_N.                                     \\
        \tilde{X}_N(0)      &= Q_N^T X_N Q_N.
\end{align*}
We assume that the numerical low-rank approximation is accurate enough such that $\mathcal{R}(X_N)\approx 0$.
Then it holds:
\begin{align*}
    \twonorm{Q_N^T \mathcal{R}(X_N) Q_N} \leq \twonorm{\mathcal{R}(X_N)}\approx 0.
\end{align*}
This means that the projected residual $Q_N ^T  \mathcal{R}(X_N) Q_N$ is even smaller than the residual of the \ARE{} $\mathcal{R}(X_N)$
and, therefore, we can neglect the residual.

    \subsection{Reduced Trial Space for the Galerkin Approach using Eigenvalue Decay}\label{sec:galerkin_approach_for_large_scale_differential_riccati_equations:galerkin_approach_using_eigenvalue_decay}

Let $X_{\infty} = Z_{\infty} Z_{\infty}^T$ be the exact solution of the~\ARE{}~\eqref{eqn:ch_adre:are}.
Moreover let $Z_{\infty} = Q_{\infty} S_{\infty} V_{\infty}^T$ be its compact singular value decomposition,
such that $Q_{\infty} \in \Rmat{n}{p},~S_{\infty} \in \Rmat{p}{p}$ and $V_{\infty} \in \Rmat{p}{p}$ and
$Z_{\infty} = Q_{\infty} S_{\infty} V_{\infty}^T$.
The compact singular value decomposition of $Z_{\infty}$ gives a spectral decomposition of
$X_{\infty}$ that is
\begin{align*}
    X_{\infty} = Z_{\infty} Z_{\infty}^T = Q_{\infty} S_{\infty}^2 Q_{\infty}^T,~\text{and}~
    S_{\infty}^2 = \diag \left( \LambdaK{1}\left( X_{\infty}\right),\ldots,  \LambdaK{p}\left( X_{\infty}\right)\right).
\end{align*}
This means that the diagonal matrix $S_{\infty}^2$ contains all non-zero eigenvalues of $X_{\infty}$ in a non-increasing fashion.
We have that $\range\left( X_{\infty} \right) = \range\left( Z_{\infty} \right) = \range\left( Q_{\infty} \right)$.
Because of Theorem~\ref{thm:ch_adre:are_range} it holds that $\range\left( Q_{\infty} \right)= \mathcal{K} \left( A^T, C^T \right)$.
According to Theorem~\ref{thm:ch_adre:dre_invariant_subspace} we can represent the solution in the following form
\begin{align*}
    X(t) = Q_{\infty} Q_{\infty}^T X(t) Q_{\infty} Q_{\infty}^T.
\end{align*}
This representation has the advantage that the entries of $Q_{\infty}^T X(t) Q_{\infty}$ can be bounded by the eigenvalues of $X_{\infty}$.

\begin{theorem}\label{thm:galerkin:entry_decay}
    Let $(A,B)$ be stabilizable and $(A,C)$ be detectable. Moreover let $X_{\infty} \in \Rmat{n}{n}$ be the unique symmetric positive semidefinite solution of the
    ARE~\eqref{eqn:ch_adre:are} and $q_1,\ldots,q_n \in \Rfield^{n}$ be a system of orthonormal eigenvectors of $X_{\infty}$ corresponding to the eigenvalues
    $\LambdaK{1} \left( X_{\infty} \right),\ldots  ,\LambdaK{n} \left( X_{\infty} \right)\in \Rfield$.
    Then for all $i,j = 1,\ldots,n $ and $t \geq 0$ the following holds:
    \begin{align}
        \abs{q_i^T X(t) q_j} \leq \sqrt{\LambdaK{i} \left( X_{\infty} \right) \LambdaK{j} \left(X_{\infty}\right)},\label{eqn:galerkin:entry_decay}
    \end{align}
    where $X$ is the unique solution of the \DRE{}~\eqref{eqn:ch_adre:auto_dre} with $X_0 = 0$.
\end{theorem}

\begin{proof}
    According to Theorem~\ref{thm:ch_adre:ale_existence} the inequality $0\preccurlyeq X(t) \preccurlyeq X_{\infty}$
    holds for all $t\geq 0$.
    By multiplying the inequality with $q_i^T$ from the left and $q_i$ from the right we obtain
    \begin{align*}
        0
        \leq
        q_{i}^T X(t) q_i \leq q_{i}^T X_{\infty} q_{i}
        =
        \LambdaK{i}\left( X_{\infty} \right)  q_{i}^T q_{i}
        =
        \LambdaK{i}\left(X_{\infty} \right).
    \end{align*}
    Now let $i\neq j$ and $\alpha, \beta \in \Rfield$.
    Again by multiplying the inequality with $\alpha q_i + \beta q_j$ we obtain
    \begin{align*}
        0
        \leq
        {\left( \alpha q_i + \beta q_j \right)}^T X(t) \left( \alpha q_i + \beta q_j \right)
        \leq
        {\left( \alpha q_i + \beta q_j \right)}^T X_{\infty} \left( \alpha q_i + \beta q_j \right).
    \end{align*}
    Since $X(t)$ and $X_{\infty}$ are symmetric, it applies that
    \begin{align*}
        \alpha^2 q_i^T X(t) q_i + 2 \alpha \beta q_i^T X(t) q_j + \beta^2 q_j^T X(t) q_j
        \leq
        \alpha^2 q_i^T X_{\infty} q_i + 2\alpha \beta q_i^T X_{\infty}q_j + \beta^2 q_j^T X_{\infty}q_j.
    \end{align*}
    As $q_i$ and $q_j$ are different orthonormal eigenvectors of $X_{\infty}$, we obtain for the right hand side
    \begin{align*}
        \alpha^2 q_i^T X_{\infty} q_i + 2\alpha \beta q_i^T X_{\infty}q_j + \beta^2 q_j^T X_{\infty}q_j
        &=
        \alpha^2 \LambdaK{i}\left(X_{\infty}\right) + 2\alpha \beta \LambdaK{j}\left(X_{\infty}\right) q_i^T q_j + \beta^2 \LambdaK{i}\left(X_{\infty}\right) \\
        &=
        \alpha^2 \LambdaK{i}\left(X_{\infty}\right) + \beta^2 \LambdaK{i}\left(X_{\infty}\right).
    \end{align*}
    As $X(t)$ is symmetric positive semidefinite, the following inequality holds for the left hand side.
    \begin{align*}
        \alpha^2 q_i^T X(t) q_i +2 \alpha \beta q_i^T X(t) q_j + \beta^2 q_j^T X(t) q_j \geq 2\alpha \beta q_i^T X(t) q_j.
    \end{align*}
    Now we have
    \begin{align*}
        0
        \leq
        \alpha^2 \LambdaK{i}\left( X_{\infty} \right) - 2 \alpha \beta q_i^T X(t) q_j + \beta^2 \LambdaK{j} \left( X_{\infty} \right)
        =
        \begin{bmatrix} \alpha & \beta \end{bmatrix}
        \begin{bmatrix} \LambdaK{i}\left( X_{\infty} \right)  & - {q_i}^T X(t)q_j \\ -{q_i}^T X(t) q_j & \LambdaK{j}\left( X_{\infty} \right) \end{bmatrix}
        \begin{bmatrix} \alpha \\ \beta \end{bmatrix}.
    \end{align*}
    Since this holds for all $\alpha,\beta \in \Rfield$ the matrix
    \begin{align*}
        \begin{bmatrix} \LambdaK{i}\left( X_{\infty} \right)  & - {q_i}^T X(t)q_j \\ -{q_i}^T X(t) q_j & \LambdaK{j}\left( X_{\infty} \right) \end{bmatrix}
    \end{align*}
    is symmetric positive semidefinite. Therefore its determinant must be non-negative,
    \begin{align*}
        0
        \leq
        \LambdaK{i}\left( X_{\infty} \right) \LambdaK{j}\left( X_{\infty} \right) - {\left( {q_i}^T X(t) q_j \right)}^2.
    \end{align*}
    Finally this leads to
    \begin{align*}
        \abs{{q_i}^T X(t) q_j}
        \leq
        \sqrt{ \LambdaK{i}\left( X_{\infty} \right) \LambdaK{j}\left( X_{\infty} \right)}.
    \end{align*}
\end{proof}

Let the columns of $Q_{\infty}$ be $q_1,\ldots,q_p$.
Due to the decay of the eigenvalues $\LambdaK{k}\left( X_\infty \right)$ of the solution of the \ARE~\eqref{eqn:ch_adre:are} and the inequality~\eqref{eqn:galerkin:entry_decay} from Theorem~\ref{thm:galerkin:entry_decay},
the values $\abs{q_i^T X(t) q_j}$ also decay for $i+j$ increasing.
We have that
\begin{align*}
    X(t)
    =
    Q_{\infty} Q_{\infty}^T X(t) Q_{\infty} Q_{\infty}^T
    =
    \sum\limits_{i,j=1}^{p} \left({q_i^T X(t)q_j}\right) q_{i} q_j^T.
\end{align*}

For quick enough eigenvalue decay, we expect that $\abs{q_i^T X(t) q_j}\leq \sqrt{\LambdaK{i} \left( X_{\infty} \right) \LambdaK{j} \left(X_{\infty}\right)}\approx 0$
for $i+j$ large enough.
We truncate the series and obtain
\begin{align*}
    X(t)
    \approx
    \sum\limits_{i,j=1}^{k} \left({q_i^T X(t)q_j}\right) q_{i} q_j^T
    =
    Q_{\infty,k} Q_{\infty,k}^T X(t) Q_{\infty,k} Q_{\infty,k}^T,
\end{align*}
where $Q_{\infty,k}=\begin{bmatrix}q_1,\ldots,q_k \end{bmatrix}\in \Rmat{n}{k}$.
We also consider the appropriate real linear space
\begin{align*}
    \mathcal Q_{\infty,k}:= \left\{ Q_{\infty,k} Y Q_{\infty,k}^T \mid Y\in \Rmat{q}{q} \right\} \subseteq \Rmat{n}{n}
\end{align*}
together with the orthogonal projection
\begin{align*}
    \mathcal P_k: \Rmat{n}{n}\to \mathcal Q_{\infty,k}, \ \mathcal P_{\infty,k}(X) = Q_{\infty,k} Q_{\infty,k}^T X Q_{\infty,k} Q_{\infty,k}^T
\end{align*}
As the columns of $Q_{\infty,k}$ are orthonormal, it holds that $\mathcal P_{\infty,k}^2(X) = \mathcal P_{\infty,k}(X)$.
Moreover the projection $\mathcal P_{\infty,k}$ is orthogonal, because
\begin{align*}
    \frodot{X-\mathcal P_{\infty,k}(X)}{Q_{\infty,k}YQ_{\infty,k}}
    &=
    \frodot{X- Q_{\infty,k}Q_{\infty,k}^T X Q_{\infty,k}Q_{\infty,k}^T}{Q_{\infty,k}YQ_{\infty,k}^T}
    \\
    &=
    \frodot{X}{Q_{\infty,k}YQ_{\infty,k}^T} - \frodot{Q_{\infty,k}Q_{\infty,k}^T X Q_{\infty,k}Q_{\infty,k}^T}{Q_{\infty,k}YQ_{\infty,k}^T}
    \\
    &=
    \frodot{X}{Q_{\infty,k}YQ_{\infty,k}^T} - \frodot{X}{Q_{\infty,k}YQ_{\infty,k}^T} = 0
\end{align*}
for all $Y\in \Rmat{k}{k}$. Therefore the best approximation of $X(t)$ in $\mathcal Q_{\infty,k}$ is given by
\begin{align*}
    \sum\limits_{i,j=1}^{k} \left({q_i^T X(t)q_j}\right) q_{i} q_j^T
    = \mathcal P_{\infty,k}(X(t))
    = \myargmin\limits_{X \in \mathcal Q_{\infty,k}} \fronorm{X-X(t)}
\end{align*}
and for the approximation error we obtain
\begin{align*}
    \fronorm{X(t) - \mathcal P_{\infty,k}(X(t))}
    &=
    \fronorm{\sum\limits_{\substack{i,j=1\\  i>k \lor j>k }}^{p} \left({q_i^T X(t)q_j}\right) q_{i} q_j^T}
    =
    \sqrt{\sum\limits_{\substack{i,j=1\\ i>k \lor j>k}}^{p} \abs{ {q_i^T X(t)q_j} }^2}
    \\
    &\leq
    \sqrt{\sum\limits_{\substack{i,j=1\\ i>k \lor j>k}}^{p} \LambdaK{i}\left(X_{\infty}\right) \LambdaK{j}\left( X_{\infty} \right) }.
\end{align*}
Since the eigenvalues $\LambdaK{p+1}\left( X_{\infty} \right),\ldots,\LambdaK{n}\left( X_{\infty} \right)$ are $0$ we obtain
\begin{align}\label{eqn:galerkin:projection-error}
    \fronorm{X(t) - \mathcal P_{\infty,k}(X(t))}
    \leq
    \sqrt{\sum\limits_{\substack{i,j=1\\ i>k \lor j>k }}^{n} \LambdaK{i}\left(X_{\infty}\right) \LambdaK{j}\left( X_{\infty} \right) }.
\end{align}

We propose therefore to setup a trial space for the Galerkin approach using a system of eigenvectors corresponding to the largest eigenvalues.
This can be obtained by using a low-rank method to obtain a numerical approximation of the solution of the \ARE{}.
Then a compact singular value decomposition of the numerical low-rank approximation of $X_{\infty}$ can be used to obtain an approximation of the eigenvectors corresponding to the largest eigenvalues.
The small singular values can be safely truncated from the singular value decomposition by virtue of Thm.~\ref{thm:galerkin:entry_decay}.
This reduces also the dimension of the trial space.
Let
\begin{align*}
    Z_N = Q_N S_N  V_N^T.
\end{align*}
be the truncated reduced singular value decomposition of the low-rank approximation.
With that, the trial space for the Galerkin approach is given by
\begin{align*}
    \left\{ Q_N  \tilde X Q_N^T \mid  \tilde X\in \Rmat{p}{p} \right\},
\end{align*}
and, as $X(t)$ converges to $X_{\infty}$ and $X_{\infty} \approx Z_N Z_N^T$, we propose the Galerkin ansatz
\begin{align*}
    X(t) \approx Z_N Z_N^T  - Q_{N}\tilde{X}(t) Q_N^T.
\end{align*}

\begin{example}[Decay of Absolute Values of  Entries]
We illustrate the decay of $\left|q_i X(t) q_j^T\right|$ in \textup{Figures~\ref{fig:ch_adre:entries_decay_dre:t1}-\ref{fig:ch_adre:entries_decay_dre:t9}}.
We have chosen the same matrices as for the \textup{Example~\ref{ex:eigenvalue_decay}}.
To improve the visualization all values below machine precision were set to machine precision.
The eigenvalue decay of the solution $X_{\infty}$ of the corresponding \ARE{} is shown in \textup{Figure~\ref{fig:ch_adre:entries_decay_dre:xinf}}.

\newcommand{\EntriesDecayPlotScale}{0.49125\textwidth}

\newcommand{\EntriesDecayScale}{1.0\textwidth}

\newcommand{\EntriesDecaySkip}{\\[0.5em]}

\newcommand{\EntriesDecayPlot}[3]{
    \begin{minipage}[h]{\EntriesDecayPlotScale}
        \captionsetup{type=figure}
        \centering
        \includegraphics[width=\EntriesDecayScale]{tex_figures/entry_decay_dre_t#1.pdf}
        \captionof{figure}{#2}\label{fig:ch_adre:entries_decay_dre:#3}
    \end{minipage}
}

\vspace{1em}
\EntriesDecayPlot{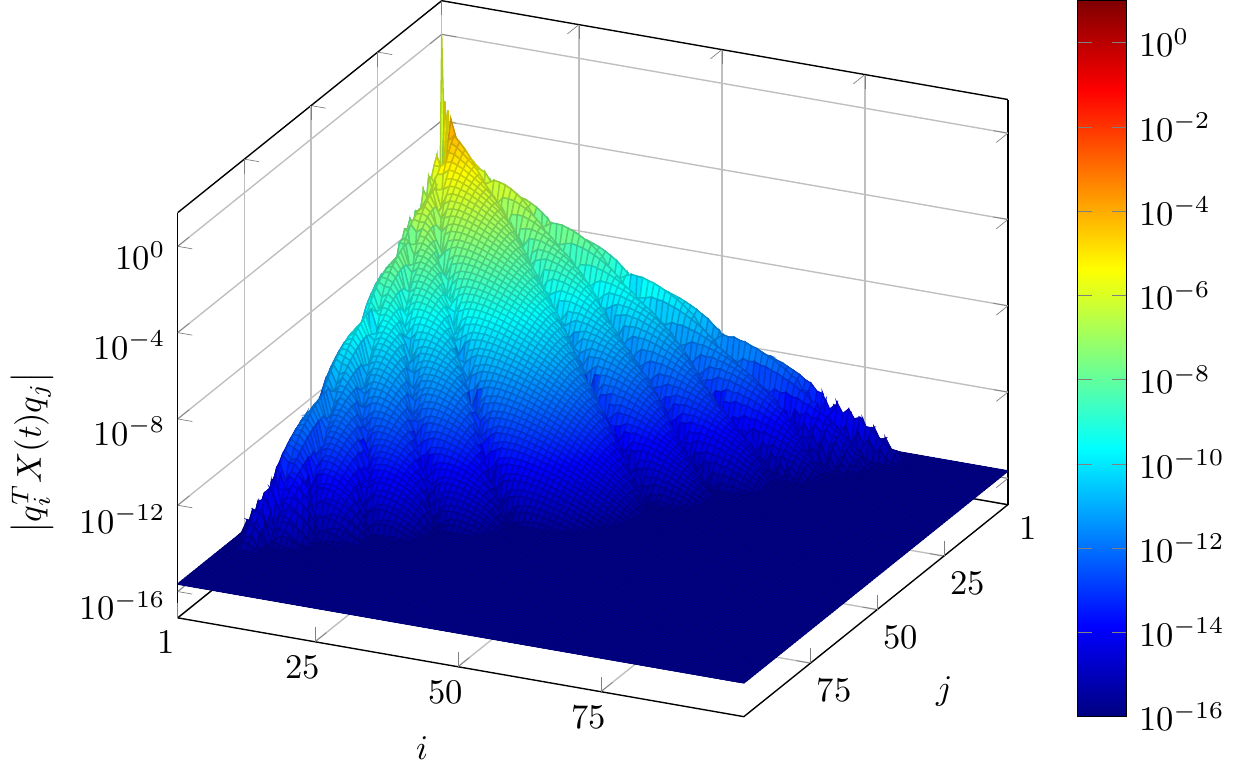}{Decay of $\abs{q_i^T X(t) q_j}$ for $t=1$.}{t1}
\hfill
\EntriesDecayPlot{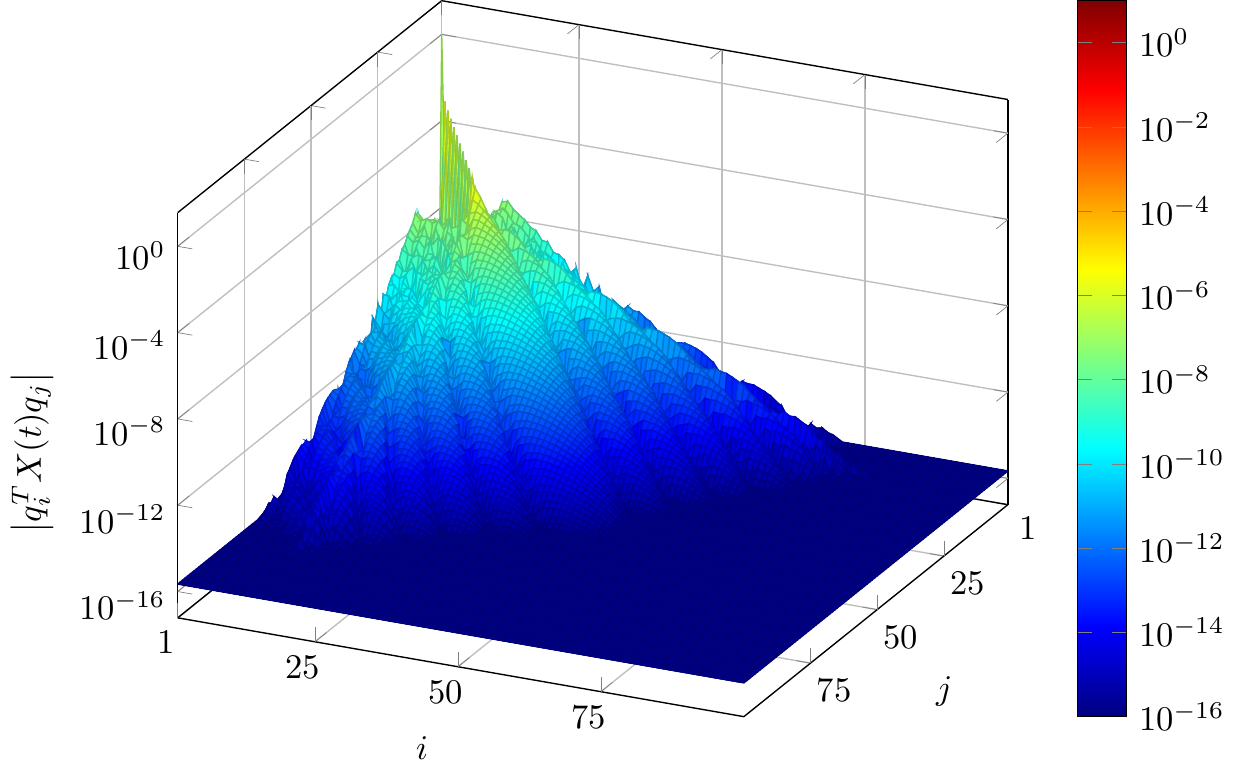}{Decay of $\abs{q_i^T X(t) q_j}$ for $t=3$.}{t3}
\EntriesDecaySkip{}
\EntriesDecayPlot{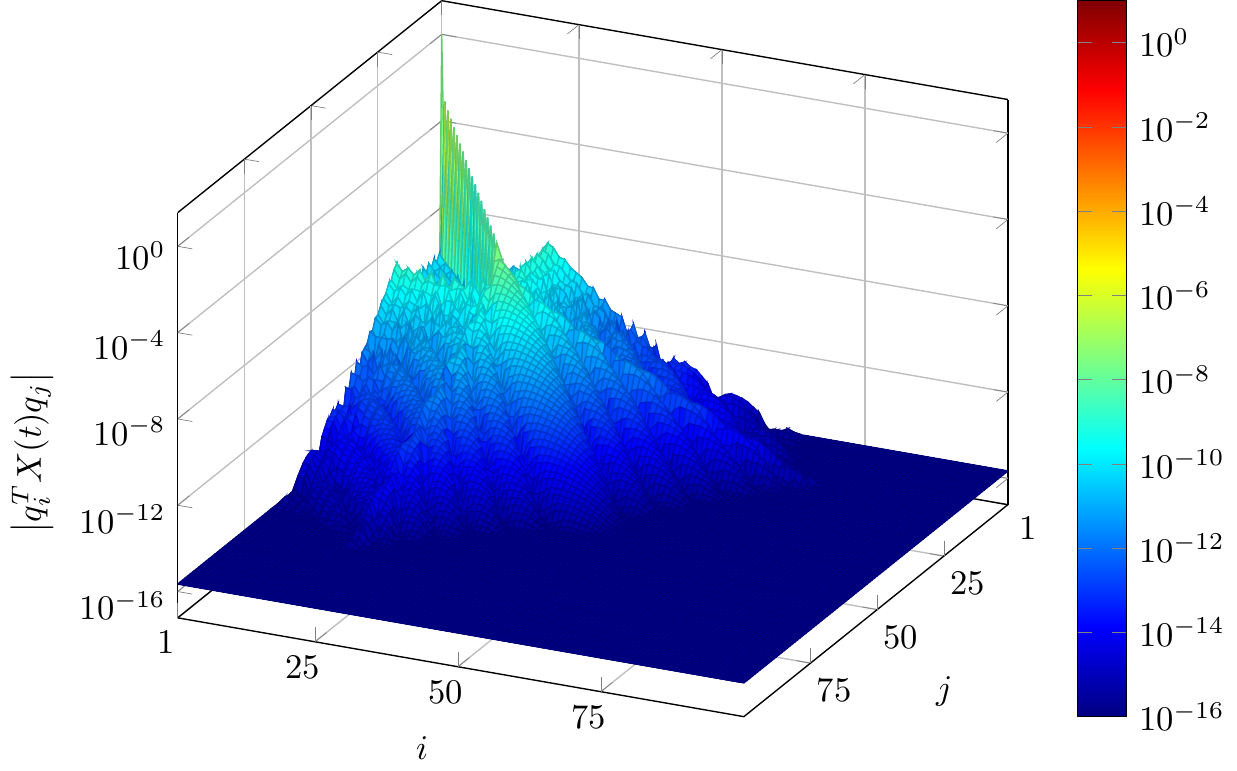}{Decay of $\abs{q_i^T X(t) q_j}$ for $t=5$.}{t5}
\hfill
\EntriesDecayPlot{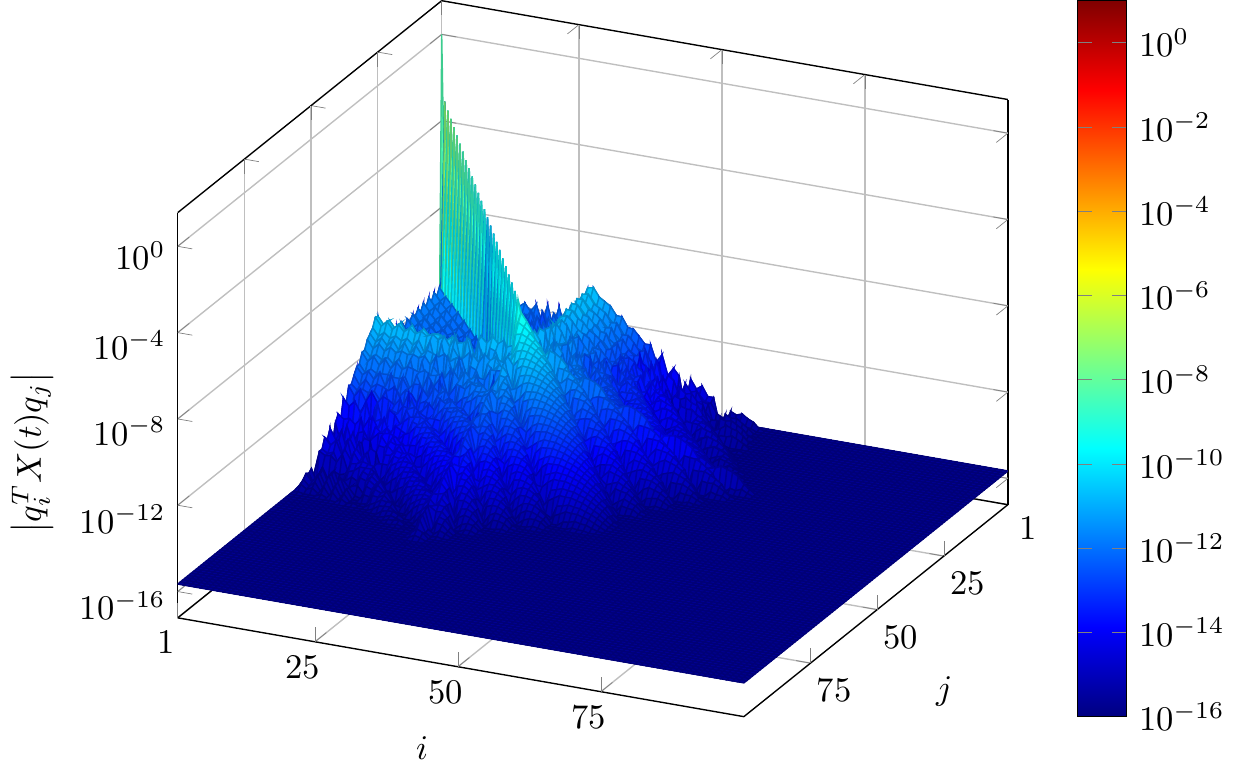}{Decay of $\abs{q_i^T X(t) q_j}$ for $t=7$.}{t7}
\EntriesDecaySkip{}
\EntriesDecayPlot{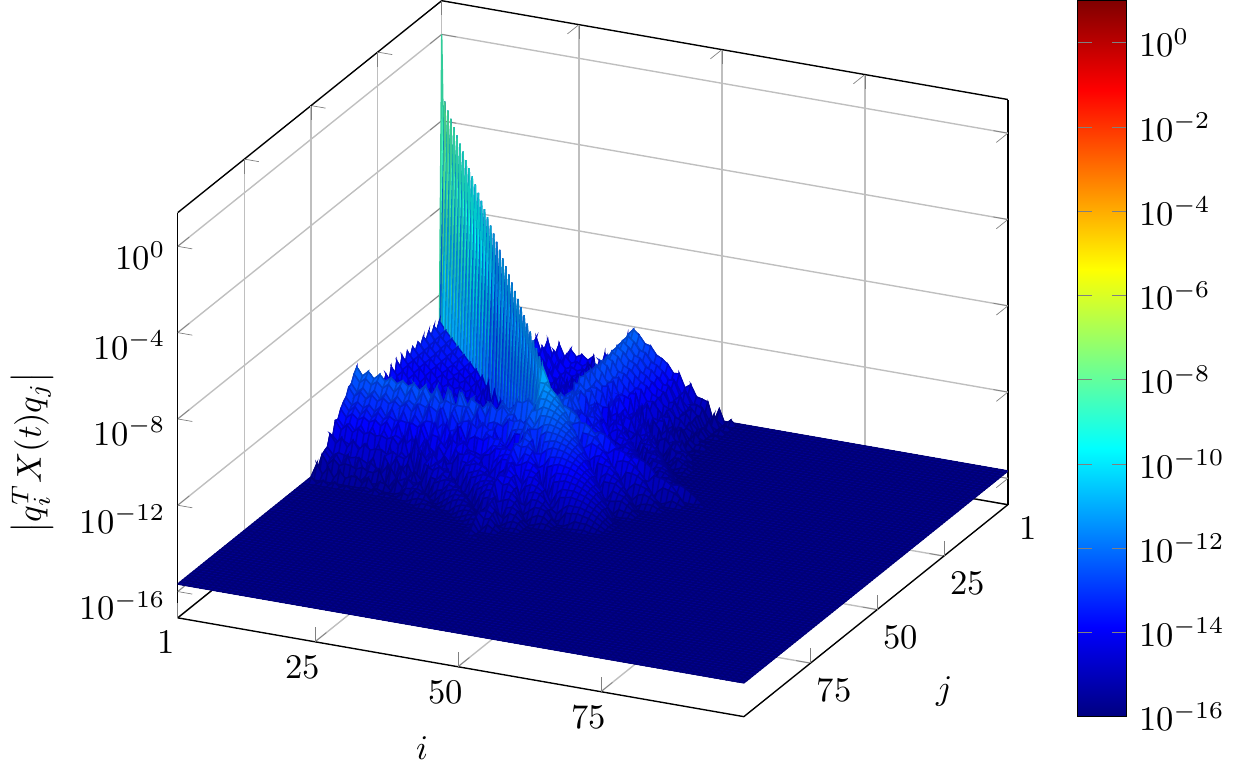}{Decay of $\abs{q_i^T X(t) q_j}$ for $t=9$.}{t9}
\begin{minipage}[h]{\EntriesDecayPlotScale}
    \captionsetup{type=figure}
    \centering
    \includegraphics[width=0.775\textwidth]{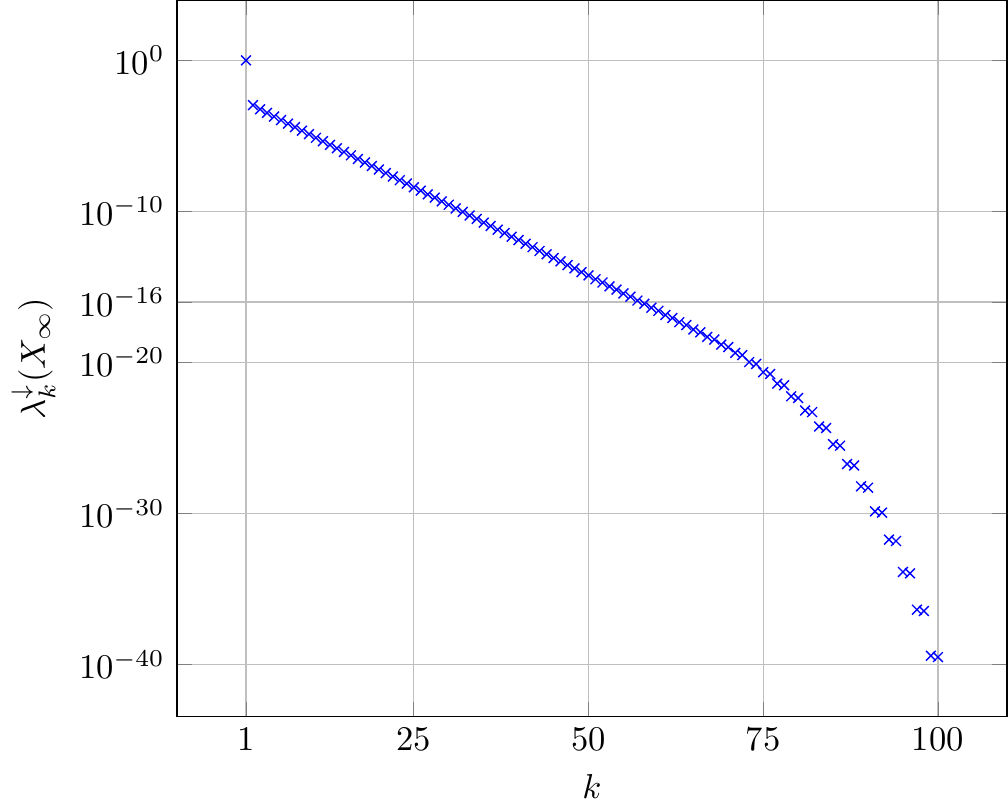}
    \captionof{figure}{The eigenvalue decay of $X_{\infty}$.}\label{fig:ch_adre:entries_decay_dre:xinf}
\end{minipage}
\end{example}

\begin{remark}
    With minor adjustments, all arguments also hold for the generalized \DRE{}
\begin{subequations}\label{eqn:galerkin:gen_auto_dre}
    \begin{align}
        M^T\dot{X}(t)M & = A^T X(t) M + {M}^T X(t)A  - M^T X(t)BB^T X(t) M + C^{T}C, \label{eqn:galerkin:gen_auto_dre_eqn}  \\
        X(0)           & = 0, \label{eqn:galerkin:gen_auto_dre_initial}
    \end{align}
\end{subequations}
    with $M\in \Rmat{n}{n}$ nonsingular
    that can accommodate, e.g., a mass matrix from a finite element discretization.
\end{remark}

In summary, the proposed approach reads as written down in Algorithm~\ref{alg:galerkin}.
\begin{algorithm}
    \caption{Galerkin approach for the generalized \DRE~\eqref{eqn:galerkin:gen_auto_dre} (\AREGALERKIN{})}\label{alg:galerkin}
    \begin{algorithmic}[1]
        \Assumption{$\left(AM^{-1}, B\right)$ is stabilizable and $(AM^{-1},CM^{-1})$ is detectable.
        }
        \Require{$M,~A\in \Rmat{n}{n},~B\in \Rmat{n}{b},~C\in \Rmat{c}{n}$.}
        \Ensure{$X(t) \approx Z_{\infty}Z_{\infty}^T - Q_{\infty} \tilde{X}(t) Q_{\infty}^T$
        that approximates the solution to \\ $M^T\dot{X}(t)M=A^T X(t) M + M^{T}X(t) A - M^T X(t) BB^T X(t) M + C^T C,\ X(0)=0$.}
        \vspacealgorithm{}
        \Statex{\% Solve the \ARE{}:}
        \State{$A^T X_{\infty}M + M^{T}X_{\infty} A  -M^T X_{\infty} BB^T X_{\infty} M + C^T C = 0$ for $X_{\infty} \approx Z_{\infty} Z_{\infty}^T$ and $Z_{\infty}\in \Rmat{n}{q}$;}
        \vspacealgorithm{}
        \Statex{\% Compute compact singular value decomposition:}
        \State{$[Q_{\infty}, S_{\infty},\sim] = \texttt{svd}(Z_{\infty},0)$;}
        \vspacealgorithm{}
        \Statex{\% Set tolerance to largest singular value times machine epsilon:}
        \State{$tol = \vareps_{\text{machine}} \cdot S_{\infty}(1,1)$;}
        \vspacealgorithm{}
        \Statex{\% Truncate all singular values smaller than tolerance and get truncated low-rank factor:}
        \State{$idx = \texttt{diag}(S_{\infty})\geq tol$;}\label{alg:galerkin:truncate}
        \State{$S_{\infty} = S_{\infty}(idx,idx)$;}
        \State{$Q_{\infty} = Q_{\infty}(:,idx)$;}
        \State{$Z_{\infty} = Q_{\infty}S_{\infty}$;}\label{alg:galerkin:Qinf}
        \vspacealgorithm{}
        \Statex{\% Compute matrices:}
        \State{$A_F = Q_{\infty}^T \left( AM^{-1} - BB^T Z_{\infty}Z_{\infty}^T\right) Q_{\infty}$;}
        \State{$B_F = Q_{\infty}^T  B$;}
        \vspacealgorithm{}
        \Statex{\% Solve the differential equation using Algorithm~\ref{alg:nonsymdre:mod_davison_maki}:}
        \State{$\dot{\tilde{X}}(t) = A_F^T \tilde{X}(t) + \tilde{X}(t) A_F + \tilde{X}(t) B_F B_F^T \tilde{X}(t),\ \tilde{X}(0) = S_{\infty}^2$;}
    \end{algorithmic}
\end{algorithm}

\pagebreak
\section{Numerical Experiments}\label{sec:numerical_experiments}

To quantify the performance of Algorithm~\ref{alg:galerkin}, we consider a number of differential Riccati equations that are used to define optimal controls.
Concretely, we consider the generalized differential Riccati equation
\begin{subequations}\label{eqn:ch_num_setup:dre}
\begin{align}
    M^T\dot{X}(t)M & = A^T X(t) M + {M}^T X(t)A  - M^T X(t)BB^T M + C^{T}C, \label{eqn:ch_num_setup:dre_eqn} \\
    X(0)           & = 0. \label{eqn:ch_num_setup:dre_ini}
\end{align}
\end{subequations}
and their realizations. First, we consider the \RAIL{} benchmark example, that is a finite element discretization of a heat equation; see~\cite{morBenS05} for the model description.
The second example, \CONVDIFF{}, derives from a finite-differences discretized heat equation with convection on the unit square with homogenous Dirichlet boundary conditions,
\begin{align*}
    \frac{\partial}{\partial t} x(\xi,t) - \Delta x(\xi,t)  - v \cdot \nabla x(\xi,t)  &= f(\xi) u(t)\quad    \text{in}\ \Omega \times (0,T)
\end{align*}
where $\Omega = {(0,1)}^2$ and $v={\left[10, 100\right]}^T$; see~\cite{Pen00a}.

On both examples, we compare the proposed method with the splitting methods developed in~\cite{Sti18a,Sti15}.
The splitting methods are based on a splitting of the \DRE{} into an affine and nonlinear subproblem. The advantages of that approach lie in the fact that the nonlinear subproblem can be solved by an explicit solution formula.
The numerical solution of the linear subproblem is based on approximating the action of a matrix exponential by means of Krylov subspace methods.
We used the \matlab{} implementation \dresplit{}~\cite{dresplit} of the splitting methods for our experiments.
In the tests, we employed the \emph{Lie} and \emph{Strang} splitting of order $1$ and $2$ respectively, as well as the symmetric splitting of order $4,6$ and $8$.
We abbreviate the methods by \LIE{}, \STRANG{}, \SYMMETRICtwo{}, \SYMMETRICfour{}, \SYMMETRICsix{} and \SYMMETRICeight{}.

To evaluate the error, we computed a reference solution $X_{\texttt{ref}}(t)$ using \SYMMETRICeight{}  with constant time step size $h$.
The basic information about the setup of the benchmark problems are given in Table~\ref{tab:benchmark}.
\begin{table}[!h]
\centering
\begin{tabular}{@{}lllll@{}}
\toprule
    problem     & $n$   & matrices                                  & interval      & reference solution                \\ \midrule
                &       & $M$ symmetric positive definite,          &               &                                   \\
    \RAIL{}     & 5177  & $A$ symmetric, $M^{-1}A$ stable,          & $[0,464]$     & \SYMMETRICeight{}, $h=2^{-5}$     \\
                &       & $B\in \Rmat{n}{6},\ C\in \Rmat{7}{n}$     &               &                                   \\ \midrule
                &       & $M=I_n$,                                  &               &                                   \\
    \CONVDIFF{} & 6400  & $A$ nonsymmetric and stable,              & $[0,0.125]$   & \SYMMETRICeight{}, $h=2^{-18}$    \\
                &       & $B\in \Rmat{n}{1},\ C\in \Rmat{1}{n}$     &               &                                   \\
\bottomrule
\end{tabular}
\caption{Information about benchmark problems.}\label{tab:benchmark}
\end{table}

All computations are carried out on a machine with 2$\times $ \xeon{} Skylake Silver 4110 @ 2.10GHz CPU with 8 cores, 192 GB Ram and \matlab{} 2018a.
We have used the low-rank Newton ADI iteration implemented in \mexmess{}\cite{mymessweb}~to solve the algebraic Riccati equations; as required for our approach as laid out in Algorithm~\ref{alg:galerkin}.

We report the absolute and relative errors
\begin{align*}
    \norm{X(t) - X_{\texttt{ref}}(t)}\quad \text{and}\quad \frac{\norm{X(t) - X_{\texttt{ref}}(t)}}{\norm{X_{\texttt{ref}}(t)}},
\end{align*}
where $X(t)$ is the numerical approximation and $X_{\texttt{ref}}(t)$ is the reference solution in $2$-norm and Frobenius norm.
We also report the norm of the reference solution $\norm{X_{\texttt{ref}}(t)}$ as well as the convergence to the stationary point
$\norm{X_{\texttt{ref}}(t) - X_{\infty}}_2$.

Numerical results for the Galerkin approximation from Algorithm~\ref{alg:galerkin} and for the splitting scheme based solvers
and be found in Appendices~\ref{appendix:numerical_results_for_galerkin_approach} and~\ref{appendix:numerical_results_for_splitting_schemes}.
The computational costs for both methods are given in Section~\ref{sec:numerical_experiments:computational_time}.
Also, we evaluate the best approximation in the trial space of the reference solution, which is given by
\begin{align*}
    X_{\texttt{best}}(t) := Q_{\infty}Q_{\infty}^{T} X_{\texttt{ref}}(t) Q_{\infty}Q_{\infty}^T = \myargmin\limits_{X \in \left\{ Q_{\infty} \tilde{X} Q_{\infty}^T  \mid \tilde{X} \in \Rmat{p}{p}\right\}} \norm{X-X_{\texttt{ref}} (t)}_F,
\end{align*}
where $Q_{\infty}$ is the matrix from Algorithm~\ref{alg:galerkin} Line~\ref{alg:galerkin:Qinf}.

The code of the implementation and the precomputed reference solution are available as mentioned in Figure~\ref{fig:linkcodndat}.
\begin{figure}[h!]
    \begin{framed}
        \textbf{Code and Data Availability} \\
        The source code of the implementations used to compute the presented results is available from:
        \begin{center}
            \href{https://doi.org/10.5281/zenodo.2629737}{\texttt{doi:10.5281/zenodo.2629737}}\\
            \href{https://gitlab.mpi-magdeburg.mpg.de/behr/behbh19_dre_are_galerkin_code}{\texttt{https://gitlab.mpi-magdeburg.mpg.de/behr/behbh19\_dre\_are\_galerkin\_code}}
        \end{center}
        under the GPLv2+ license
        and is authored by Maximilian Behr.
    \end{framed}
    \caption{Link to code and data.}\label{fig:linkcodndat}
\end{figure}

    \subsection{Galerkin Approach and Splitting Schemes}\label{sec:numerical_experiments:galerkin_approach_and_splitting_schemes}


The initial step of Algorithm~\ref{alg:galerkin} requires the solution to the associated~\ARE{}.
For this task we call \mexmess{} that iteratively computes the numerical solution to the following absolute and relative residuals
\begin{align*}
    \norm{A^T Z_{\infty} Z_{\infty}^T  M + M^T  Z_{\infty} Z_{\infty}^T  A - M^T  Z_{\infty} Z_{\infty}^T  B B^T Z_{\infty} Z_{\infty}^T  M + C^T C}_2
    \intertext{and}
    \frac{\norm{A^T Z_{\infty} Z_{\infty}^T  M + M^T  Z_{\infty} Z_{\infty}^T  A - M^T  Z_{\infty} Z_{\infty}^T  B B^T Z_{\infty} Z_{\infty}^T   M + C^T C}_2}
    {\norm{{C}^T C}_2}.
\end{align*}

The achieved values for the different test setups as well as the number of columns of the corresponding $Z_{\infty}$ after truncation (see Step~\ref{alg:galerkin:truncate} of Algorithm~\ref{alg:galerkin}),
that define the dimension of the reduced model, are listed in Table~\ref{fig:ch_num_exp_galerkin:residual}.

\begin{table}[h]
    \centering
    \includegraphics{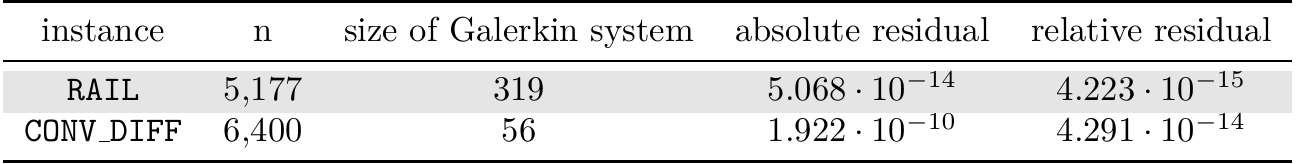}
    \caption{Residuals for the \ARE{} $0 = A^T X M + M X A - M^T X B B ^T X M + C^T C$.}\label{fig:ch_num_exp_galerkin:residual}
\end{table}

The $1$-norm bound for the matrix exponential $tol_{\texttt{exp}}$ from Algorithm~\ref{alg:nonsymdre:mod_davison_maki} was set to $1\cdot 10^{10}$.
The resulting step sizes are given in Table~\ref{fig:ch_num_exp_galerkin:step_sizes}.

\begin{table}[H]
    \centering
    \includegraphics{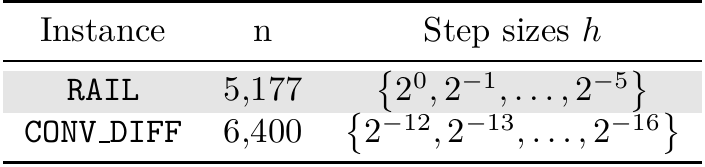}
    \caption{Step sizes $h$ for modified Davison-Maki method Algorithm~\ref{alg:nonsymdre:mod_davison_maki}.}\label{fig:ch_num_exp_galerkin:step_sizes}
\end{table}

We plot the numerical errors in Figures
\ref{fig:appendix:galerkin:galerkin_plots:RAIL:5177:k0:abs}--\ref{fig:appendix:galerkin:galerkin_plots:RAIL:5177:k-5:rel} and
\ref{fig:appendix:galerkin:galerkin_plots:CONV_DIFF3:80:k-12:abs}--\ref{fig:appendix:galerkin:galerkin_plots:CONV_DIFF3:80:k-16:rel}.
Figures
\ref{fig:appendix:norm_reference_solution:RAIL:5177}, \ref{fig:appendix:convergence_stationary:RAIL:5177},
\ref{fig:appendix:norm_reference_solution:CONV_DIFF3:80} and \ref{fig:appendix:convergence_stationary:CONV_DIFF3:80}
show the norm of the reference solution and the convergence to the stationary point.

In view of the performance, we can interpret the presented numbers and plots as follows:
Firstly, the accuracy of the \emph{modified Davison-Maki method}; cf. Figure~\ref{fig:appendix:galerkin:galerkin_plots:RAIL:5177:k0:rel} and~\ref{fig:appendix:galerkin:galerkin_plots:RAIL:5177:k-5:rel} is independent of the step size,
as discussed in Section~\ref{sec:radons_lemma:davison_maki_methods}. Still we compute the solution on different time grids, since for control applications the values of the solution might be needed at many time instances.

The computational times for \AREGALERKIN{} include the solve of the corresponding \ARE{} and the subsequent integration of the projected dense \DRE{}.
Since the efforts for the time integration exactly doubles with a bisection of the step size, from the timings for the \RAIL{} problem, with, e.g., $42$s ($h=2^{-3}$) and  $77$s ($h=2^{-4}$) (see Figure~\ref{fig:num_exp_computational_time:splitting:plots_time:RAIL_5177_GALERKIN}), one infers that most of the time is spent to solve the dense \DRE{}.
Conversely, for the \CONVDIFF{} benchmark problem, most of the time ($45$s) was used to solve the \ARE{}. As the resulting Galerkin projected \DRE{} system is of size $56$ only, the computational costs for the time integration are vanishingly small. Accordingly, the differences in the effort caused by finer time grids are hardly visible; see Figure \ref{fig:num_exp_computational_time:splitting:plots_time:CONV_DIFF3_80_GALERKIN}.

The reference solution for the \RAIL{} problem is large in norm what makes the absolute error comparatively large;
see the Figure~\ref{fig:appendix:norm_reference_solution:RAIL:5177} in Appendix~\ref{appendix:numerical_results_for_galerkin_approach}.

In both examples, in terms of accuracy, the \AREGALERKIN{} approximation is nearly at the same level as the high order splitting schemes,
cf. Figures~\ref{fig:appendix:galerkin:galerkin_plots:RAIL:5177:k0:rel},~\ref{fig:appendix:galerkin:splitting_plots:RAIL:5177:k-4:rel}
and
Figures~\ref{fig:appendix:galerkin:galerkin_plots:CONV_DIFF3:80:k-12:rel},~\ref{fig:appendix:galerkin:splitting_plots:CONV_DIFF3:80:k-16:rel}.
We note, however, that the~\AREGALERKIN{} method does not give the best possible approximation in the trial space; compare the error levels for $X_{\mathsf{best}}$.

In any case, the~\AREGALERKIN{} method clearly outperforms the splitting methods in terms of computational time versus accuracy in all test examples.

    \subsection{Computational Time}\label{sec:numerical_experiments:computational_time}


\newcommand{\TimingMinipageScale}{0.50\textwidth}

\newcommand{\TimingScale}{0.95\textwidth}

\newcommand{\TimingCaptionAREGALERKIN}{Timing for \AREGALERKIN{}.}
\newcommand{\TimingCaptionSplitting}{Timing for splitting schemes.}

\newcommand{\TimingPlots}[2]{
    \begin{minipage}{\TimingMinipageScale}
        \captionsetup{type=figure}
        \centering
        \includegraphics[width=\TimingScale]{tex_figures/plots_time_#1.pdf}
        \captionof{figure}{#2}\label{fig:num_exp_computational_time:splitting:plots_time:#1}
    \end{minipage}
}



\newcommand{\TimingSkip}{\\[0.5em]}

\newcommand{\TimingPlotsLegend}[1]{
    \begin{minipage}[c]{\textwidth}
        \centering
        \includegraphics{tex_figures/plots_time_legend_#1.pdf}
    \end{minipage}
}

\TimingPlots{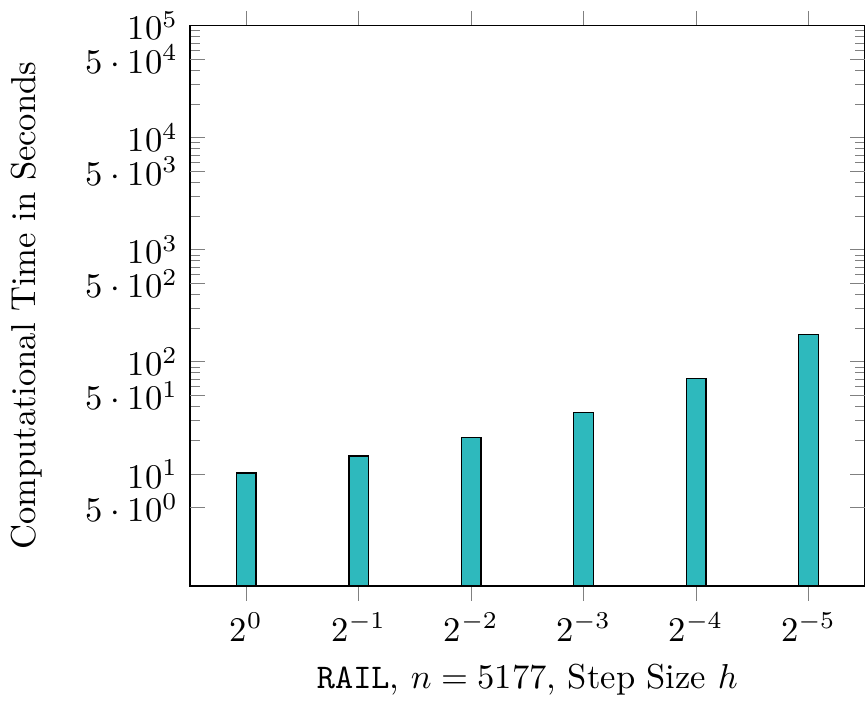}{\TimingCaptionAREGALERKIN{}}
\TimingPlots{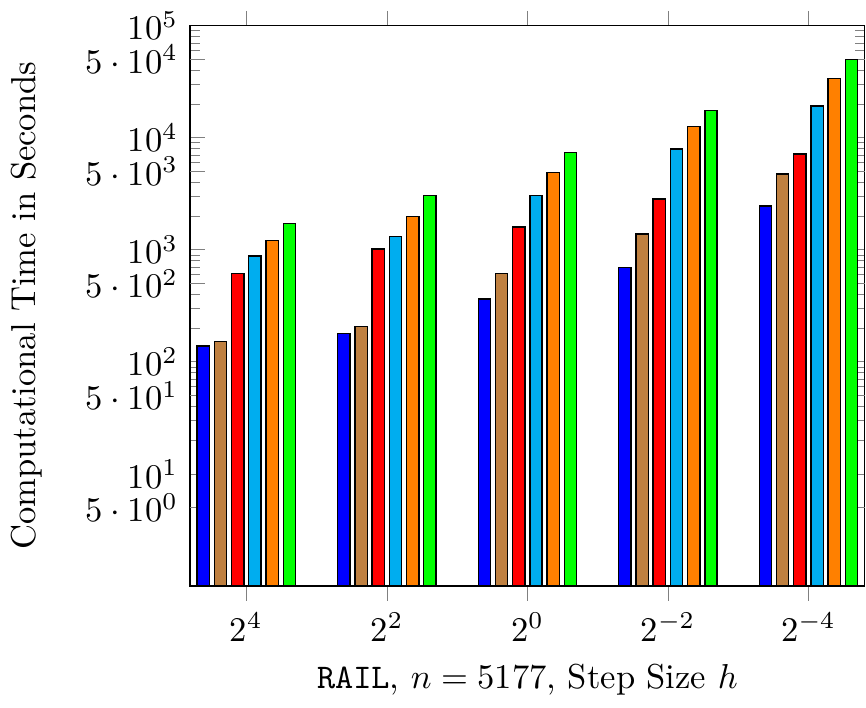}{\TimingCaptionSplitting{}}

\TimingPlots{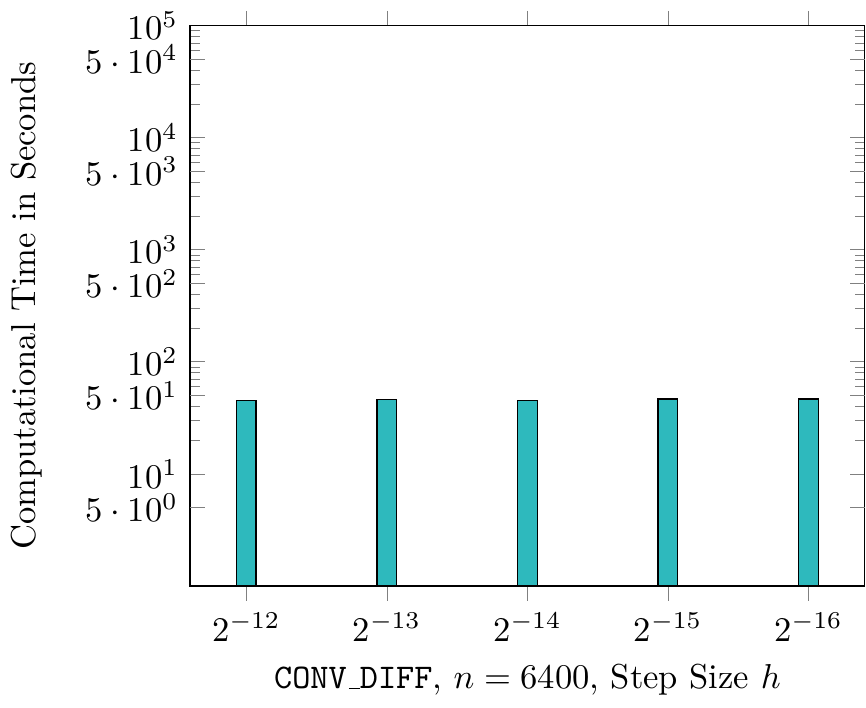}{\TimingCaptionAREGALERKIN{}}
\TimingPlots{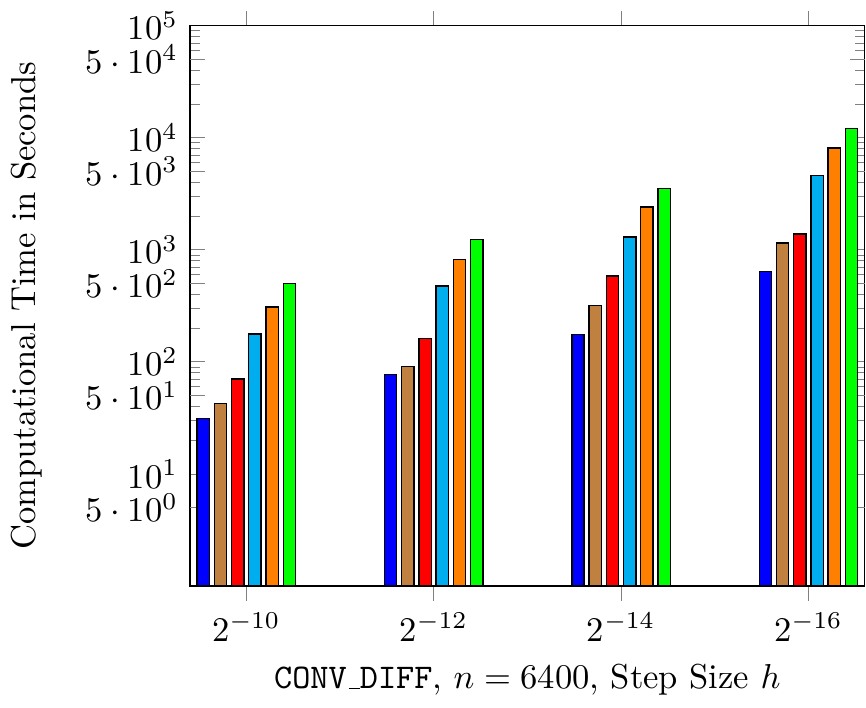}{\TimingCaptionSplitting{}}
\TimingSkip{}
\TimingPlotsLegend{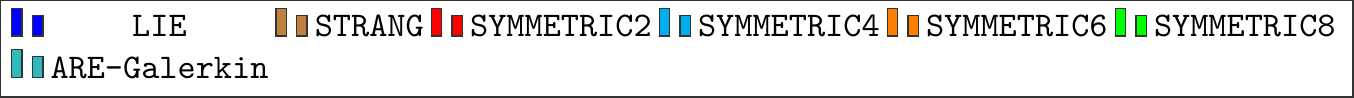}



\pagebreak
\section{Conclusion}\label{sec:conclusion}

We have reviewed and extended fundamental properties of the solution to the differential and algebraic Riccati equation
and heavily relied on the solution representation provided by Radon's Lemma to analyze variants of \emph{Davison-Maki methods}
and to derive an efficient Galerkin projection scheme.
Numerical tests confirmed that the resulting projected scheme outperforms existing methods in terms of computation time, memory requirements, and approximation quality.
In particular, storage requirements have been the bottleneck in the numerical considerations of large-scale differential Riccati equations.

Our proposed Galerkin ansatz bases on a low-rank approximation of the associated algebraic Riccati equation (\ARE{}) for which there are efficient solvers.
Moreover, the information on the residual and on eigenvalue decay,
that come with the low-rank iteration for the ARE can be directly transferred into estimates for the approximation quality
of our approach the more that the use of the \emph{Davison-Maki methods} leads to an \emph{exact} time discretization.

Future work will deal with the treatment of nonzero initial conditions. While the formulas are easily extended to this case, the invariance properties and the eigenvalue comparisons, that were the backbone of our numerical approach, are no longer given in general.
For the (not so) special case that the initial condition $X_0$ writes as $X_0= C^{T}WC$ with a symmetric positive definite weighting matrix $W$, the flow invariance as established
in Section~\ref{sec:galerkin_approach_for_large_scale_differential_riccati_equations:galerkin_approach_using_solution_representations} still holds so that the presented algorithm can be applied without modification.
For a general low-rank initial condition $X_0 = Z_0Z_0^T$ the flow invariance can be achieved by taking the columns of the solution to the \ARE{}~\eqref{eqn:ch_adre:are} with $C^{T}C$ replaced by ${\left[C^T, Z_0\right]\left[C^T, Z_0\right]}^T$ as the Galerkin ansatz space.
In any case, the inequality $0\preccurlyeq X(t) \preccurlyeq X_\infty$, where $X_\infty$ is the solution of the \ARE{}, does not hold anymore and, thus, approximation quality cannot be assured as in~\eqref{eqn:galerkin:projection-error}.
However, if one can find a matrix $\tilde X$ for which the comparison $0\preccurlyeq X(t) \preccurlyeq \tilde X$ holds, all arguments of Section~\ref{sec:galerkin_approach_for_large_scale_differential_riccati_equations:galerkin_approach_using_eigenvalue_decay} apply accordingly.

\bibliography{mybib/mybib}
\pagebreak

\appendix
\section{Numerical Results for Galerkin Approach}\label{appendix:numerical_results_for_galerkin_approach}


\newcommand{\GalerkinMinipageScale}{0.45\textwidth}

\newcommand{\GalerkinScale}{0.875\textwidth}

\newcommand{\GalerkinPlots}[5]{
    \begin{minipage}{\GalerkinMinipageScale}
        \captionsetup{type=figure}
        \centering
        \includegraphics[width=\GalerkinScale]{tex_figures/galerkin_plots_#1_#2_k#3_#4.pdf}
        \captionof{figure}{#5}\label{fig:appendix:galerkin:galerkin_plots:#1:#2:k#3:#4}
    \end{minipage}
}

\newcommand{\GalerkinCaptionAbs}{Absolute error of the Galerkin and Best approximation.}
\newcommand{\GalerkinCaptionRel}{Relative error of the Galerkin and Best approximation.}

\newcommand{\GalerkinSkip}{\\[0.5em]}

\newcommand{\GalerkinPlotsLegend}{
    \begin{minipage}[c]{\textwidth}
        \centering
        \includegraphics{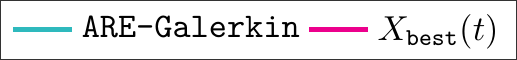}
    \end{minipage}
}

\newcommand{\ReferenceSolutionPlots}[2]{
    \begin{minipage}{\GalerkinMinipageScale}
        \captionsetup{type=figure}
        \centering
        \includegraphics[width=\GalerkinScale]{tex_figures/norm_reference_solution_#1_#2.pdf}
        \captionof{figure}{Norm of the reference solution.}\label{fig:appendix:norm_reference_solution:#1:#2}
    \end{minipage}
}

\newcommand{\ConvergenceStationaryPlot}[2]{
    \begin{minipage}{\GalerkinMinipageScale}
        \captionsetup{type=figure}
        \centering
        \includegraphics[width=\GalerkinScale]{tex_figures/convergence_stationary_#1_#2.pdf}
        \captionof{figure}{Convergence to the stationary point.}\label{fig:appendix:convergence_stationary:#1:#2}
    \end{minipage}
}

\begin{center}\texttt{RAIL}, $n=5177$ and $M^T\dot{X}(t)M = A^T X(t) M +  M^T X(t) A - M^T X(t) B B ^T X(t) M + C^T C,\ X(0)=0$. \end{center}
\GalerkinPlots{RAIL}{5177}{0}{abs}{\GalerkinCaptionAbs{}}
\hfill
\GalerkinPlots{RAIL}{5177}{0}{rel}{\GalerkinCaptionRel{}}
\GalerkinSkip{}
\GalerkinPlots{RAIL}{5177}{-5}{abs}{\GalerkinCaptionAbs{}}
\hfill
\GalerkinPlots{RAIL}{5177}{-5}{rel}{\GalerkinCaptionRel{}}
\GalerkinSkip{}
\GalerkinPlotsLegend{}
\ReferenceSolutionPlots{RAIL}{5177}
\hfill
\ConvergenceStationaryPlot{RAIL}{5177}
\pagebreak

\begin{center}\texttt{CONV\_DIFF}, $n=6400$ and $\dot{X}(t) = A^T X(t)  +   X(t) A - X(t) B B ^T X(t) + C^T C,\ X(0)=0$. \end{center}
\GalerkinPlots{CONV_DIFF3}{80}{-12}{abs}{\GalerkinCaptionAbs{}}
\hfill
\GalerkinPlots{CONV_DIFF3}{80}{-12}{rel}{\GalerkinCaptionRel{}}
\GalerkinSkip{}
\GalerkinPlots{CONV_DIFF3}{80}{-16}{abs}{\GalerkinCaptionAbs{}}
\hfill
\GalerkinPlots{CONV_DIFF3}{80}{-16}{rel}{\GalerkinCaptionRel{}}
\GalerkinSkip{}
\GalerkinPlotsLegend{}
\ReferenceSolutionPlots{CONV_DIFF3}{80}
\hfill
\ConvergenceStationaryPlot{CONV_DIFF3}{80}

\pagebreak

\section{Numerical Results for Splitting Schemes}\label{appendix:numerical_results_for_splitting_schemes}


\newcommand{\SplittingMinipageScale}{0.45\textwidth}

\newcommand{\SplittingScale}{0.875\textwidth}

\newcommand{\SplittingPlots}[5]{
    \begin{minipage}{\SplittingMinipageScale}
        \captionsetup{type=figure}
        \centering
        \includegraphics[width=\SplittingScale]{tex_figures/splitting_plots_#1_#2_k#3_#4.pdf}
        \captionof{figure}{#5}\label{fig:appendix:galerkin:splitting_plots:#1:#2:k#3:#4}
    \end{minipage}
}

\newcommand{\SplittingCaptionAbs}{Absolute error of the splitting scheme \mbox{approximation}.}
\newcommand{\SplittingCaptionRel}{Relative error of the splitting scheme \mbox{approximation}.}

\newcommand{\SplittingSkip}{\\[0.5em]}

\newcommand{\SplittingPlotsLegend}{
    \begin{minipage}[c]{\textwidth}
        \centering
        \includegraphics{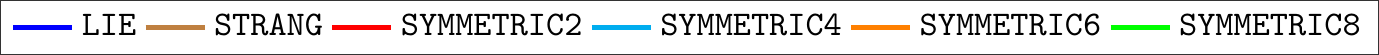}
    \end{minipage}
}

\begin{center}\texttt{RAIL}, $n=5177$ and $M^T\dot{X}(t)M = A^T X(t) M +  M^T X(t) A - M^T X(t) B B ^T X(t) M + C^T C,\ X(0)=0$. \end{center}
\SplittingPlots{RAIL}{5177}{4}{abs}{\SplittingCaptionAbs{}}
\hfill
\SplittingPlots{RAIL}{5177}{4}{rel}{\SplittingCaptionRel{}}
\SplittingSkip{}

\SplittingPlots{RAIL}{5177}{0}{abs}{\SplittingCaptionAbs{}}
\hfill
\SplittingPlots{RAIL}{5177}{0}{rel}{\SplittingCaptionRel{}}
\SplittingSkip{}

\SplittingPlots{RAIL}{5177}{-4}{abs}{\SplittingCaptionAbs{}}
\hfill
\SplittingPlots{RAIL}{5177}{-4}{rel}{\SplittingCaptionRel{}}
\SplittingSkip{}
\SplittingPlotsLegend{}
\pagebreak

\begin{center}\texttt{CONV\_DIFF}, $n=6400$ and $\dot{X}(t) = A^T X(t)  +   X(t) A - X(t) B B ^T X(t) + C^T C,\ X(0)=0$. \end{center}
\SplittingPlots{CONV_DIFF3}{80}{-10}{abs}{\SplittingCaptionAbs{}}
\hfill
\SplittingPlots{CONV_DIFF3}{80}{-10}{rel}{\SplittingCaptionRel{}}
\SplittingSkip{}

\SplittingPlots{CONV_DIFF3}{80}{-14}{abs}{\SplittingCaptionAbs{}}
\hfill
\SplittingPlots{CONV_DIFF3}{80}{-14}{rel}{\SplittingCaptionRel{}}
\SplittingSkip{}

\SplittingPlots{CONV_DIFF3}{80}{-16}{abs}{\SplittingCaptionAbs{}}
\hfill
\SplittingPlots{CONV_DIFF3}{80}{-16}{rel}{\SplittingCaptionRel{}}
\SplittingSkip{}
\SplittingPlotsLegend{}
\pagebreak

\end{document}